\documentclass[a4paper,11pt,reqno]{amsart}
\usepackage[T1]{fontenc}
\usepackage[utf8]{inputenc}
\usepackage{lmodern}
\usepackage{amsmath, amsthm, amssymb,eucal, amscd, mathrsfs, amsfonts}
\usepackage{hyperref}
\usepackage{pxfonts}
\usepackage{yhmath}

\usepackage{graphicx}
\usepackage{color}

\usepackage{xcolor}
\usepackage{todonotes}

\usepackage[all]{xy}

\usepackage{latexsym}
\usepackage{pict2e}
\usepackage{graphicx}

\usepackage[lite]{amsrefs}
\newcommand*{\MRref}[2]{ \href{http://www.ams.org/mathscinet-getitem?mr=#1}{MR \textbf{#1}}}
\newcommand*{\arxiv}[1]{\href{http://www.arxiv.org/abs/#1}{arXiv: #1}}
\renewcommand{\PrintDOI}[1]{\href{http://dx.doi.org/\detokenize{#1}}{doi: \detokenize{#1}}%
  \IfEmptyBibField{pages}{, (to appear in print)}{}}

\usepackage{amssymb}

\def\og{\leavevmode\raise.3ex\hbox{$\scriptscriptstyle\langle\!\langle$~}}
\def\fg{\leavevmode\raise.3ex\hbox{~$\!\scriptscriptstyle\,\rangle\!\rangle$}}

\usepackage{hyperref}
\usepackage{array}

\usepackage{nicefrac}

\newtheorem{thm}{Theorem}[section]
\newtheorem{pro}[thm]{Proposition}
\newtheorem{lem}[thm]{Lemma}
\newtheorem{cor}[thm]{Corollary}

\theoremstyle{definition}
\newtheorem{df}[thm]{Definition}

\newtheorem{dflem}[thm]{Definition and Lemma}

\theoremstyle{remark}
\newtheorem{rem}[thm]{Remark}

\newtheorem{ex}[thm]{Example}

\allowdisplaybreaks

\newcommand\id{1}
\newcommand{\Id}{\operatorname{Id}}
\newcommand\Hom{\operatorname{Hom}}
\newcommand{\cstar}{\text{C}^*}

\usepackage[all]{xy}

\def\commutatif{\ar@{}[rd]|{\circlearrowleft}}

\newcommand{\eq}[1][r]
   {\ar@<-3pt>@{-}[#1]
    \ar@<-1pt>@{}[#1]|<{}="gauche"
    \ar@<+0pt>@{}[#1]|-{}="milieu"
    \ar@<+1pt>@{}[#1]|>{}="droite"
    \ar@/^2pt/@{-}"gauche";"milieu"
    \ar@/_2pt/@{-}"milieu";"droite"}
    
    \def\dar[#1]{\ar@<2pt>[#1]\ar@<-2pt>[#1]}
  \entrymodifiers={!!<0pt,0.7ex>+} 

\newcommand{\cA}{{\mathcal A}}  \newcommand{\sA}{{\mathscr A}}
\newcommand{\cB}{{\mathcal B}}  \newcommand{\sB}{{\mathscr B}}
  \newcommand{\sC}{{\mathscr C}}
  \newcommand{\sD}{{\mathscr D}}
\newcommand{\cE}{{\mathcal E}}  \newcommand{\sE}{{\mathscr E}}
\newcommand{\cF}{{\mathcal F}}   
\newcommand{\cG}{{\mathcal G}}
\newcommand{\cH}{{\mathcal H}}

\newcommand{\cK}{{\mathcal K}}
\newcommand{\cL}{{\mathcal L}}  
\newcommand{\cM}{{\mathcal M}}

\newcommand{\cT}{{\mathcal T}}
\newcommand{\cU}{{\mathcal U}}

\newcommand{\cZ}{{\mathcal Z}}

\newcommand{\CC}{{\mathbb C}}
\newcommand{\EE}{{\mathbb E}}
\newcommand{\FF}{{\mathbb F}}

\newcommand{\NN}{{\mathbb N}}

\newcommand{\RR}{{\mathbb R}}

\newcommand{\uc}{{\mathbb S}^1}

\newcommand{\ZZ}{{\mathbb Z}}

\newcommand{\bfS}{{\mathbf S}}

\newcommand{\wGa}{{\widetilde \Gamma}}

\newcommand{\wPU}{\what{\operatorname{PU}}}

\newcommand{\wU}{\what{\operatorname{U}}}

\newcommand{\K}{\operatorname{K}}
\newcommand{\KK}{\operatorname{KK}}
\newcommand{\KR}{\operatorname{KR}}
\newcommand{\KKR}{\operatorname{KKR}}

\newcommand{\wRExt}{{\what{\operatorname{\textrm{Ext}R}}}}

\newcommand{\wRBr}{{\what{\operatorname{\textrm{Br}R}}}}

\newcommand{\EER}{\operatorname{\mathbf{ER}}}

\newcommand{\spin}{\operatorname{Spin}}

\newcommand{\pin}{\operatorname{Pin}}

\newcommand{\grpd}{\xymatrix{\cG \dar[r] & X}}
\newcommand{\gamgpd}{\xymatrix{\Ga \dar[r] & Y}}

\newcommand{\what}{\widehat}

\newcommand{\fr}{{\mathfrak r}}
\newcommand{\fs}{{\mathfrak s}}

\newcommand{\RG}{{\mathfrak {RG}}}

\newcommand{\frc}{\mathfrak{c}}

\newcommand{\To}{\longrightarrow}

\newcommand{\mto}{\longmapsto}

\newcommand{\Ga}{\Gamma}

\newcommand{\del}{\delta}
\newcommand{\al}{\alpha}
\newcommand{\be}{\beta}

\newcommand{\ve}{\varepsilon}
\newcommand{\vp}{\varphi}
\newcommand{\g}{\gamma}

\newcommand{\Om}{\Omega}
\newcommand{\om}{\omega}
\newcommand{\Lam}{\Lambda}

\newcommand{\Isom}{\operatorname{Isom}}
\newcommand{\Aut}{\operatorname{\textsf{Aut}}}

\newcommand{\Cl}{\CC l}

\newcommand{\Co}{{\mathcal C}_0}
\newcommand{\Cor}{\mathfrak{Corr}}

\def\<{\langle}
\def\>{\rangle}
\let\ipscriptstyle=\scriptscriptstyle
\def\lipsqueeze{{\mskip -3.0mu}}
\def\ripsqueeze{{\mskip -3.0mu}}
\def\ipcomma{\nobreak\mathrel{,}\nobreak}
\newbox\ipstrutbox
\setbox\ipstrutbox=\hbox{\vrule height8.5pt
width 0pt}
\def\ipstrut{\copy\ipstrutbox}
\def\lip#1<#2,#3>{\mathopen{\relax_{\ipstrut\ipscriptstyle{
#1}}\lipsqueeze
\langle} #2\ipcomma #3 \rangle}
\def\blip#1<#2,#3>{\mathopen{\relax_{\ipstrut
\ipscriptstyle{ #1}}\lipsqueeze\bigl\langle} #2\ipcomma #3 \bigr\rangle}
\def\rip#1<#2,#3>{\langle #2\ipcomma #3
\rangle_{\ripsqueeze\ipstrut\ipscriptstyle{#1}}}
\def\brip#1<#2,#3>{\bigl\langle #2\ipcomma #3
\bigr\rangle_{\ripsqueeze\ipstrut\ipscriptstyle{#1}}}
\def\angsqueeze{\mskip -6mu}
\def\smangsqueeze{\mskip -3.7mu}
\def\trip#1<#2,#3>{\langle\smangsqueeze\langle #2\ipcomma #3
\rangle\smangsqueeze\rangle_{\ripsqueeze\ipstrut\ipscriptstyle{#1}}}
\def\btrip#1<#2,#3>{\bigl\langle\angsqueeze\bigl\langle #2\ipcomma
#3
\bigr\rangle
\angsqueeze\bigr\rangle_{\ripsqueeze\ipstrut\ipscriptstyle{#1}}}
\def\tlip#1<#2,#3>{\mathopen{\relax_{\ipstrut\ipscriptstyle{
#1}}\lipsqueeze \langle\smangsqueeze\langle} #2\ipcomma #3
\rangle\smangsqueeze\rangle}
\def\btlip#1<#2,#3>{\mathopen{\relax_{\ipstrut\ipscriptstyle{
#1}}\lipsqueeze
\bigl\langle\angsqueeze\bigl\langle} #2\ipcomma #3
\bigr\rangle\angsqueeze\bigr\rangle}

\def\ip(#1|#2){(#1\mid #2)}
\def\bip(#1|#2){\bigl(#1 \mid #2\bigr)}
\def\Bip(#1|#2){\Bigl( #1 \bigm| #2 \Bigr)}
\def\h[#1,#2]{[#1,#2]_{H}}

\newcommand{\mydot}{\mathbin{:}}

\def\ipp(#1|#2){\ip({#1}|{#2})_{\pi}}

\hypersetup{
  colorlinks = true,
  urlcolor = blue,
  linkcolor = blue,
  citecolor = red,
  pdfauthor = {Moutuou, E.},
  pdfkeywords = {Twisted K-theory, KK-theory},
  pdftitle = {Moutuou, E. - Thom Isomorphism in twisted groupoid K-theory},
  pdfsubject = {Twisted K-theory},
  pdfpagemode = UseNone
}

\title[Thom isomorphism in Twisted groupoid $\K$-Theory]{Equivariant $\KK$-theory for generalised actions and Thom isomorphism in groupoid twisted $\K$-theory}
\author{El-ka\"ioum M. Moutuou}
\address{Institut \'Elie Cartan de Lorraine - Metz, Universit\'e de Lorraine et CNRS, Ile de Saulcy, 57000 Metz}
\curraddr{School of Mathematics, University of Southampton, Highfield, Southampton, SO17 1BJ, UK}
\email{E.MohamedMoutuou@soton.ac.uk}



\begin{document}

\maketitle

\begin{abstract}
We develop equivariant $\KK$--theory for locally compact groupoid actions by Morita equivalences on real and complex graded $\cstar$-algebras. Functoriality with respect to generalised morphisms and Bott periodicity are discussed. We introduce Stiefel-Whitney classes for real or complex equivariant vector bundles over locally compact groupoids to establish the Thom isomorphism theorem in twisted groupoid $\K$--theory. 
\end{abstract}

\tableofcontents

\section{Introduction}

Equivariant $\KK$-theory for groupoid actions by automorphisms was introduced by Le Gall in~\cite{LeGall:KK_groupoid} generalising Kasparov's work in equivariant $\KK$-theory for groups~\cite{Kasparov:Novikov}. A locally compact groupoid $\grpd$, with source and target maps $s$ and $r$, is said to \emph{act by automorphisms} on a $\cstar$-algebra $A$ if there is a non-degenerate ${}^\ast$--homomorphism from $C_0(X)$ to the centre  $\cZ\cM(A)$ of the multiplier algebra of $A$ ($A$ is then called a $C(X)$-algebra) and an isomorphism of $C(X)$-algebras $\al:s^\ast A\To r^\ast A$ such that the induced ${}^\ast$-isomorphisms $\al_g:A_{s(g)}\To A_{r(g)}$ satisfy $\al_g\circ \al_h=\al_{gh}$ whenever $g$ and $h$ are composable. In that case $A$ is called a $\cG$-algebra~\cite{LeGall:KK_groupoid}. Given two $\cG$-algebras $A$ and $B$, Le Gall has defined a group $\KK_\cG(A,B)$ which is functorial in $B$, cofunctorial in $A$ and in $\cG$ with respect to generalised morphisms. These groups are the main setting for the study of the Baum-Connes conjecture for groupoids~\cite{Tu:BC_groupoids} and play a prominent role in twisted $\K$-theory of Lie groupoids, or more generally, of differentiable stacks~\cite{Tu:Twisted_Poincare,Tu-Xu-Laurent-Gengoux:Twisted_K}. \

In this article we weaken Le Gall's equivariant $\KK$--theory by developing from scratch a nice construction which arose in~\cite{Tu:Twisted_Poincare,Tu-Xu-Laurent-Gengoux:Twisted_K}. More precisely, we consider groupoid actions by Morita equivalences, also called {\it generalised actions}, instead of actions by automorphisms. Moreover, this is done in the general framework of "real" graded $\cstar$-algebras (complex $\ZZ_2$-graded $\cstar$-algebras endowed with $^*$-involutions compatible with the grading~\cite{Kasparov:Operator_K}) that combine real and complex $\ZZ_2$-graded $\cstar$-algebras. Recall~\cite{Moutuou:Real.Cohomology} that an involution on a topological groupoid $\cG$ is a groupoid isomorphism $\tau:\cG\To \cG$ such that $\tau^2=\id$. In that case, we say that $\cG$ is a "real" groupoid. A generalised action of a locally compact "real" groupoid $\grpd$ on a "real" graded $\cstar$-algebra $A$ consists of a graded Fell bundle with involution (~\cite{Kumjian:Fell,Moutuou:Thesis} $\cA\To \cG$ such that $C_0(X;\cA_{|X})\cong A$, where $C_0(X;\cA_{|X})$ is the algebra of sections vanishing at infinity of the restriction of $\cA$ over the unit space $X$.  Given two "real" graded $\cstar$-algebras $A$ and $B$ endowed with generalised $\cG$-actions, a group $\KKR_\cG(A,B)$ is then defined by means of equivariant \emph{correspondences} and \emph{connections}~\cite{Cuntz-Skandalis:Cones, Kasparov:Novikov}. Our main interest in exploring this theory is to establish a Thom isomorphism theorem in twisted $\K$-theory of groupoids with involutions. \\

We start by quickly giving in section 1 preliminary notions and tools used throughout the paper. {\it Correspondences} between "real" graded $\cstar$-algebras are discussed in section 2. The definition and the usual first properties of the equivariant $\KK$-theory $\KKR_\cG(A,B)$ for "real" graded $\cstar$-algebras acted upon by a locally compact second countable groupoid $\grpd$ with involution are given in section 3. In section 4, we prove that the groups $\KKR_\cG(A,B)$ are cofunctorial with respect to $\cG$ in the category $\RG$ whose objects are locally compact second-countable "real" groupoids and whose morphisms are Hilsum-Skandalis (or generalised) morphisms with involutions. In section 5 we discuss the notion of equivariant $\KKR$-equivalence for generalised actions. We use the functorial property in the category $\RG$ to establish Bott periodicity. In section 7, we introduce Stiefel-Whitney classes $w(V)$ of an equivariant Eucludean vector bundle $V$ over a locally compact "real" groupoid. These classes correspond to elements in the "real" graded Brauer group defined in~\cite{Moutuou:Brauer_Group} where we have associated to any locally compact Hausdorff second-countable "real" groupoid $\grpd$ the abelian group $\wRBr(\cG)$ generated by Morita equivalence classes of \emph{"real" graded Dixmier-Douady bundles} (\emph{i.e}. complex graded Dixmier-Douady bundles~\cite{Tu:Twisted_Poincare} with involutions) over $\cG$. Now, let $V\To X$ be a $\cG$-equivariant Euclidean vector bundle with involution locally induced from that of the "real" space $\RR^{p,q}$ (this is the space $\RR^p\times \RR^q$ equipped with the involution $(x,y)\mto (x,-y)$); $V$ is called a \emph{"real" Euclidean vector bundle of type $p-q$} over $\grpd$. Associated to such a bundle, there are two different "real" graded Dixmier-Douady bundles $\Cl(V)$ (the \emph{Clifford bundle} associated to $V$) and $\cA_V$ over $\cG$. We compare the classes of $\cA_V$ and $\Cl(V)$ in the Brauer group by working on the cohomological picture of $\wRBr(\cG)$. Specifically, it is shown in~\cite{Moutuou:Brauer_Group} that for a groupoid with involution $\tau:\cG\To \cG$, there is an isomorphism of abelian groups 
\[
DD: \wRBr(\cG) \overset{\cong}{\To}\check{H}R^0(\cG_\bullet,\ZZ_8)\oplus \check{H}R^1(\cG_\bullet,\ZZ_2)\ltimes \check{H}R^2(\cG_\bullet,\uc),
\]
where $\check{H}R^\ast$ is an equivariant groupoid cohomology with respect to involutions, and where $\uc$ is endowed with the involution given by complex conjugation and $\ZZ_2$ and $\ZZ_8$ are given the trivial involutions (see~\cite{Moutuou:Real.Cohomology}). If $DD(\cA)=(n,\al,\be)$, we say that $\cA$ is a "real" graded D-D bundle \emph{of type} $n \mod 8$. We then compare the values of $\Cl(V)$ and $\cA_V$ on the right hand side of the above isomorphism by using the \emph{Stiefel-Whitney classes}. More precisely, we show that (Theorem~\ref{thm:Calcul-de-Cl(V)})
\[
DD(\Cl(V)=(q-p,w(V))=(q-p,DD(\cA_V).
\]
 
This result along with the functorial properties in equivariant $\KKR$--theory for generalised actions are crucial for proving the Thom isomorphism theorem (Theorem~\ref{thm:Thom-isomorphism}) in the twisted $\K$-theory of locally compact groupoids with involutions, which is the purpose of section 8. Recall~\cite{Moutuou:Thesis} that if $\cG$ is a "real" groupoid, then associated to any representative $\cA$ of a class $[\cA]$ in the "real" graded Brauer group $\wRBr(\cG)$, there is a "real" graded $\cstar$-algebra $\cA\ltimes_r\cG$. The $[\cA]$--twisted $\KR$--groups of $\cG$ are defined as $\KR_\cA^{-n}(\cG^\bullet):=\KR_n(\cA\ltimes_r\cG)$~\cite{Kasparov:Operator_K}. We prove that for a "real" Euclidean vector bundle $V$ of type $p-q$ over $\grpd$ with projection $\pi:V\To X$, and a "real" graded Dixmier-Douady bundle $\cA$, there is an isomorphism of abelian groups 
\[
\KR^\ast_{\pi^\ast\cA}((\pi^\ast\cG)^\bullet)=\KR^\ast_{\cA+\Cl(V)}(\cG^\bullet),
\]
where $\pi^\ast\cG\rightrightarrows V$ is the pull-back of the groupoid $\grpd$ along $\pi$. Moreover, if $V$ is $\KR$--oriented, that is $w(V)=0$, then $\KR^\ast_{\pi^\ast\cA}((\pi^\ast\cG)^\bullet)=\KR^{\ast-p+q}(\cG^\bullet)$. As far as we know, this result was known only in the case of twistings by \emph{Azymaya} bundles over topological spaces~\cite{Donovan-Karoubi, Karoubi:Twisted} and in the case of twisted $\K$-theory of \emph{bundle gerbes}~\cite{Carey-Wang:Thom}.


\section{Preliminaries and conventions}

1.1. A \emph{"real" graded} \ $\cstar$-algebra is a $\cstar$-algebra $A$ endowed with a $\ZZ_2$-grading $A=A_0\oplus A_1$ and an involution $\tau:A\To A$ of $\cstar$-algebras compatible with the grading. In other words, $A$ is the complexification of a graded real (real in the usual sense) $\cstar$-algebra. More generally, "real" graded Banach spaces (or algebras) are defined in a similar way. It is a straightforward result that the category of "real" graded $\cstar$-algebras is isomorphic to the category of graded real $\cstar$-algebras under the operation of taking the complexification of a graded real $\cstar$-algebra in one way, and the "\emph{realification}" of a "real" graded $\cstar$-algebra (its invariant part under the involution), in the other. It follows that any theory built upon the category of "real" graded $\cstar$-algebras applies naturally to both the category of graded complex $\cstar$-algebras and the category of graded real $\cstar$-algebras. In view of this observation, G. Kasparov focused attention on $\cstar$-algebras with involutions when proving fundamental results in $\KK$-theory in his pioneering paper~\cite{Kasparov:Operator_K}. If $A$ is a "real" graded $\cstar$-algebra, a graded Hilbert $A$-module (say left) $E$ is "real" graded if it is a "real" graded Banach spaces with the property that its involution satisfies $\overline{a\cdot e}=\bar{a}\cdot \bar{e}$, and $\overline{\<e,f\>}=\<\bar{e},\bar{f}\>$ for all $a\in A$ and $e,f \in E$.\\

1.2. Recall~\cite{Moutuou:Real.Cohomology} that a groupoid $\grpd$ is \emph{"real"} if it is endowed with a groupoid isomorphism $\tau:\cG\To \cG$ such that $\tau^2=\id$. A \emph{"real" graded Fell bundle} (resp. \emph{u.s.c. Fell bundle}, where "u.s.c." stands for {\it upper semi-continuous}) over $\cG$ is a Fell bundle (resp. u.s.c Fell bundle)~\cite{Kumjian:Fell, Tu-Xu-Laurent-Gengoux:Twisted_K} $\pi:\cE\To \cG$ such that every fibre $\cE_g$ is a graded complex Banach space, and there is an involution $\cE\ni e\mto \bar{e}\in \cE$ compatible with the grading on the fibres and satisfying $\tau(\pi(e))=\pi(\bar{e})$. Note that if $\cE\To \cG$ is a "real" graded (u.s.c.) Fell bundle, then for all $g\in \cG$, the graded complex Banach space $\cE_g$ is a graded Morita $\cE_{r(g)}, \cE_{s(g)}$ --equivalence. The reduced $\cstar$-algebra $\cstar_r(\cG;\cE)$ (~\cite{Kumjian:Fell})associated to a "real" graded Fell bundle is naturally equipped with the structure of "real" graded $\cstar$-algebra. We refer to~\cite{Moutuou:Thesis} for more details on "real" graded Fell bundles and their $\cstar$-algebras. \\

1.3. Throughout the paper, all our $\cstar$-algebras, Hilbert modules, (u.s.c.) Fell bundles, continuous field of $\cstar$-algebras or Banach spaces, Dixmier-Douady bundles, are assumed "real" graded, unless otherwise stated. All our groupoids are "real", locally compact, Hausdorff and second-countable, unless otherwise stated. Thus, to avoid annoying repetitions, we will often omit the terms ""real"" and ""real" graded". We will often simply write $\cG, \Ga$, etc., for the "real" groupoids $\grpd, \gamgpd$, etc. The source and target maps of our groupoids will systematically be denoted by $s$ and $r$, respectively. Moreover, we assume the reader is familiar with the language of groupoids, groupoid actions on spaces, (graded) extensions, and groupoid cohomology, which is contained in many articles~\cite{Kumjian-Muhly-Renault-Williams:Brauer,Tu-Xu-Laurent-Gengoux:Twisted_K, Tu:Twisted_Poincare,Moutuou:Thesis}. Unless otherwise stated, all our groupoids are assumed Hausdorff, second countable, and locally compact.

\section{Equivariant $\cstar$-correspondences}

Let $A$ and $B$ be $\cstar$-algebras. A \emph{$\cstar$-correspondence} from $A$ to $B$ is a pair $(\cE,\vp)$ where $\cE$ is a Hilbert (right) $B$-module, and $\vp:A\To \cL(\cE)$ is a non-degenerate homomorphism of $\cstar$-algebras. We then view $\cE$ as a "real" graded left $A$-module by $a\cdot e:=\vp(a)e$. When there is no risk of confusion we will write ${}_A\cE_B$ for the $\cstar$-correspondence $(\cE,\vp)$. We also say that $\cE$ is a ("real" graded) $A,B$-correspondence.

If $(\cE,\vp)$ and $(\cF,\psi)$ are $\cstar$-correspondences from $A$ to $B$ and from $B$ to $C$, respectively, we define the $\cstar$-correspondence $(\cF,\psi)\circ(\cE,\vp)$ from $A$ to $C$ by $(\cE\hat{\otimes}_\psi\cF,\vp\hat{\otimes}\id)$; this $\cstar$-correspondence is called the \emph{composition of $(\cF,\psi)$ by $(\cE,\vp)$}.

An isomorphism of $\cstar$-correspondences from ${}_A\cE_B$ to ${}_A\cF_B$ is a "real" degree $0$ unitary $u\in \cL_B(\cE,\cF)$ such that $u\circ \vp_\cE(a)=\vp_\cF(a)\circ u$ for all $a\in A$. 	

\begin{df}
Let $\cG$ be a groupoid, and $A$ a $\cstar$-algebra. A \emph{generalised $\cG$-action on $A$} consists of a u.s.c. Fell bundle $\sA\To \cG$ such that $A\cong C_0(X;\sA)$, where, as usual, we have denoted $C_0(X;\sA)$ for $C_0(X;\sA_{|X})$.	
\end{df}

\begin{ex}
If $\cA\To X$ is a u.s.c. field of $\cstar$-algebras, then the u.s.c. Fell bundle $s^\ast\cA\To \cG$ is a generalised $\cG$-action on $A=C_0(X;\cA)$.	
\end{ex}

\begin{rem}
As mentioned in~\cite[\S 6.2]{Tu-Xu-Laurent-Gengoux:Twisted_K}, a generalised action is in fact an action by Morita equivalences, which justifies the terminology. Indeed, if $\sA$ is a generalised $\cG$-action on $A$, then from the properties of Fell bundles we see that for $g\in \cG$, $\sA_{g^{-1}}$ is a graded $\sA_{s(g)}$-$\sA_{r(g)}$-Morita equivalence.	
\end{rem}

Denote by $i:\cG\To \cG$ the inversion map. If $A$ is a "real" graded $\cstar$-algebra endowed with a generalized $\cG$-action $\sA$, we define $\cF_b(i^\ast\sA)$ as the "real" graded Banach algebra of norm-bounded continuous functions vanishing at infinity $a':\cG\ni g\mto a'_g\in \sA_{g^{-1}}$; the "real" structure is given by $(\bar{a'})_g:=\overline{(a'_{\bar{g}})}$, and the grading is inherited from that of $\sA$. Observe that $\cF_b(i^\ast\sA)$ is naturally a "real" graded (right) Hilbert $r^\ast A$-module under the module structure
	$$(a'\cdot a)_g :=a'_g\cdot a_g, \quad \text{for} \ a'\in \cF_b(i^\ast\sA), a\in r^\ast A=C_0(\cG;r^\ast (\sA_X)),$$ and the graded scalar product
	$$\<a',a"\>_g := (a'_g)^\ast\cdot a"_g, \quad a',a"\in \cF_b(i^\ast\sA).$$
Also, $\cF_b(i^\ast\sA)$ has the structure of "real" graded $s^\ast A$-module by setting $$(\xi\cdot a')_g:=\xi(g)\cdot a'_g, \quad {\rm for} \ \xi\in s^\ast A, a'\in \cF_b(i^\ast\sA), \ \text{and} \  g\in \cG.$$

Suppose now that $(\cE,\vp)$ is a $\cstar$-correspondence from $A$ to $B$, and $\sA$ and $\sB$ are generalised $\cG$-actions on the "real" graded $\cstar$-algebras $A$ and $B$, respectively. Then, it is easy to check that we have two $\cstar$-correspondences ${}_{s^\ast A}(\cF_b(i^\ast\sA)\hat{\otimes}_{r^\ast A}r^\ast\cE)_{r^\ast B}$ and  ${}_{s^\ast A}(s^\ast\cE\hat{\otimes}_{s^\ast B}\cF_b(i^\ast\sB))_{r^\ast B}$ with respect to the maps 
\[\Id\hat{\otimes}r^\ast\vp:s^\ast A\To \cL_{r^\ast B}(\cF_b(i^\ast\sA)\hat{\otimes}_{r^\ast A}r^\ast\cE),\] and 
\[s^\ast\vp\hat{\otimes}\Id:s^\ast A\To \cL_{r^\ast B}(s^\ast\cE\hat{\otimes}_{s^\ast B}\cF_b(i^\ast\sB)),\]
respectively. The link between these two correspondences "measures" the $\cG$-\emph{equivariance} of ${}_A\cE_B$. In particular, we give the following definition.

\begin{df}
Let $A, B, \sA, \sB$, $\cE$ be as above. A $\cstar$-correspondence ${}_A\cE_B$ is said $\cG$-\emph{equivariant} if there is an isomorphism of $\cstar$-correspondences 
\[W\in \cL(s^\ast\cE\hat{\otimes}_{s^\ast B}\cF_b(i^\ast\sB), \cF_b(i^\ast \sA)\hat{\otimes}_{r^\ast A}r^\ast\cE),\] such that for every $(g,h)\in \cG^{(2)}$, the following diagram commutes

\begin{equation}~\label{eq:equiv-corresp}
	\xymatrix{\cE_{s(h)}\hat{\otimes}_{B_{s(h)}}\sB_{h^{-1}}\hat{\otimes}_{B_{r(h)}}\sB_{g^{-1}} \ar[rr]^{W_h\hat{\otimes}_{B_{r(h)}}\Id_{\sB_{g^{-1}}}} \eq[d] && \sA_{h^{-1}}\hat{\otimes}_{A_{r(h)}}\cE_{s(g)}\hat{\otimes}_{B_{s(g)}}\sB_{g^{1}} \ar[d]^{\Id_{\sA_{h^{-1}}}\hat{\otimes}_{A_{r(g)}}W_g} \\
	\cE_{s(gh)}\hat{\otimes}_{B_{s(gh)}}\sB_{h^{-1}g^{-1}} \ar[rd]^{W_{gh}} && \sA_{h^{-1}}\hat{\otimes}_{A_{r(h)}}\sA_{g^{-1}}\hat{\otimes}_{A_{r(g)}}\cE_{r(g)} \eq[ld]\\
	& \sA_{h^{-1}g^{-1}}\hat{\otimes}_{A_{r(gh)}}\cE_{r(gh)}}	
	\end{equation}	
where the isomorphisms $\sA_{h^{-1}}\hat{\otimes}_{A_{r(h)}}\sA_{g^{-1}}\cong \sA_{h^{-1}g^{-1}}$ and $\sB_{h^{-1}}\hat{\otimes}_{B_{s(g)}}\sB_{g^{-1}}\cong \sB_{h^{-1}g^{-1}}$ come from the properties of Fell bundles. 	
\end{df}

\begin{lem}
Let $A,B$, and $C$ be $\cstar$-algebras endowed with generalized $\cG$-actions $\sA,\sB$, and $\sC$, respectively. If ${}_A\cE_B$, and ${}_B\cF_C$ are $\cG$-equivariant $\cstar$-correspondences, their composition $\cF\circ \cE$ is a $\cG$-equivariant "real" graded $A,C$-correspondence. Therefore, there is a category $\Cor_\cG$ whose objects are "real" graded $\cstar$-algebras endowed with generalized $\cG$-actions, and whose morphisms are isomorphism classes of equivariant correspondences.
\end{lem}

\begin{proof}
Suppose $W'\in \cL_{r^\ast B}(s^\ast\cE\hat{\otimes}_{s^\ast B}\cF_b(i^\ast\sB),\cF_b(i^\ast\sA)\hat{\otimes}_{r^\ast A}r^\ast\cE)$ is an isomorphism of "real" graded $s^\ast A,r^\ast B$-correspondences and $W''\in \cL_{r^\ast C}(s^\ast\cF\hat{\otimes}_{s^\ast C}\cF_b(i^\ast\sC),\cF_b(i^\ast \sB)\hat{\otimes}_{r^\ast B}r^\ast\cF)$ is an isomorphism of $s^\ast B,r^\ast C$-correspondences implementing $\cG$-equivarience. We define the isomorphism of "real" graded $s^\ast A,r^\ast C$-correspondences 
\[W\mydot s^\ast(\cE\hat{\otimes}_\psi \cF)\hat{\otimes}_{s^\ast C}\cF_b(i^\ast \sC)\To \cF_b(i^\ast\sA)\hat{\otimes}_{r^\ast A}r^\ast(\cE\hat{\otimes}_\psi \cF) \]
by setting $W\mydot=(W'\hat{\otimes}_{r^\ast B}\Id_{r^\ast\cF})\circ (\Id_{s^\ast\cE}\hat{\otimes}_{s^\ast B}W'')$, via the identification
\begin{align*}
	s^\ast(\cE\hat{\otimes}_\psi\cF) & \cong s^\ast\cE\hat{\otimes}_{s^\ast B}s^\ast\cF	\\
	r^\ast(\cE\hat{\otimes}_\psi\cF) & \cong r^\ast\cE\hat{\otimes}_{r^\ast B}r^\ast \cF.
	\end{align*}
Now it is straightforward that commutativity of the diagram~\eqref{eq:equiv-corresp} holds for $W$.	
		
\end{proof}


\section{The $\KKR_\cG$-bifunctor}

To define the equivariant $\KKR$-groups, we  need some more notions (cf.~\cite[Appendix A]{Cuntz-Skandalis:Cones}, ~\cite[Definition 6.5]{Tu-Xu-Laurent-Gengoux:Twisted_K}). Let $A$, $B$ be "real" graded $\cstar$-algebras. Let $\cE_1$ be "real" graded Hilbert $A$-module, and $\cE_2$ a "real" graded $A,B$-correspondence. Put $\cE=\cE_1\hat{\otimes}_A\cE_2$. For $\xi\in \cE_1$, let $T_\xi\in \cL_B(\cE_2,\cE)$ be given by $T_\xi(\eta):=\xi\hat{\otimes}_A\eta$ (with adjoint given by $T^\ast_\xi(\xi_1\hat{\otimes}_A\eta):=\<\xi,\xi_1\>\eta$). Observe that $\overline{T_\xi}=T_{\bar{\xi}}$, so that $T_\xi$ is "real" if and only if $\xi$ is.

Now let $A,B$, $\cE_1$, and $\cE_2$ be as above. Let $F_2\in \cL(\cE_2)$, and $F\in \cL(\cE)$. We say that $F$ is an \emph{$F_2$-connection for $\cE_1$} if for every $\xi\in \cE_1$:
\begin{align*}
	T_\xi F_2-(-1)^{|\xi|\cdot |F_2|}FT_\xi & \in \cK(\cE_2,\cE),\\
	F_2T^\ast_\xi-(-1)^{|\xi|\cdot|F_2|}T^\ast_\xi F & \in \cK(\cE,\cE_2).	
	\end{align*}

\begin{rem}~\label{rem:F-connection}
It is easy to check that $F$ is an $F_2$-connection for $\cE_1$ if and only if for every $\xi\in \cE_1$, 
\[[\theta_\xi,F_2\oplus F] \in \cK(\cE_2\oplus \cE), \]
where $\theta_\xi:=\begin{pmatrix}0&T^\ast_\xi\\ T_\xi&0\end{pmatrix}\in \cL_B(\cE_2\oplus \cE)$ (~\cite[Definition 8]{Skandalis:Remarks_KK}).	
\end{rem}

\begin{df}~\label{df:Kas-corresp}
Let $A$ and $B$ be Rg $\cstar$-algebras endowed with generalized "real" $\cG$-actions. A \emph{$\cG$-equivariant (or just equivariant if $\cG$ is understood) Kasparov $A,B$-correspondence} is a pair $(\cE,F)$ where $\cE$ is a $\cG$-equivariant Rg $A,B$-correspondence, $F$ is a "real" operator of degree $1$ in $\cL(\cE)$ such that for all $a\in A$,
\begin{itemize}
	\item[(i)] \ \ $a(F-F^\ast)\in \cK(\cE)$;
	\item[(ii)] \ \ $a(F^2-1)\in \cK(\cE)$;
	\item[(iii)] \ \ $[a,F]\in \cK(\cE)$;
	\item[(iv)] \ \ $W(s^\ast F\hat{\otimes}_{s^\ast B}\Id)W^*\in \cL(\cF_b(i^*\sA)\hat{\otimes}_{r^*A}r^*\cE)$ is an $r^*F$-connection for $\cF_b(i^*\sA)$.	
	\end{itemize}
We say that $(\cE,F)$ is \emph{degenerate} if the elements $$a(F-F^\ast),a(F^2-1),[a,F], \ {\rm and\ } [\theta_\xi,r^\ast F\oplus W(s^\ast F\hat{\otimes}_{s^\ast B}\Id)W^\ast]$$ are $0$ for all $a\in A, \xi\in \cF_b(i^\ast \sA)$.			
\end{df}

\begin{rem}
Notice that conditions (i)-(iii) are the usual ones presented in any text about $\KKR$-theory. To digest condition (iv), suppose $\cG$ acts on $A$ and $B$ by automorphisms and ${}_A\cE_B$ is a $\cG$-equivariant $A,B$-correspondence. Then the isomorphism $W$ induces a continuous family of graded isomorphisms $\widetilde{W}_g:\cE_{s(g)}\To \cE_{r(g)}$ via the identifications 
$$\cE_{s(g)}\hat{\otimes}_{\sB_{s(g)}}\sB_{r(g)}\cong \cE_{s(g)}\hat{\otimes}_{\sB_{s(g)}}\sB_{s(g)}\cong \cE_{s(g)}, \quad {\rm and } \quad \cA_{r(g)}\hat{\otimes}_{\sA_{r(g)}}\cE_{r(g)}\cong \cE_{r(g)}.$$
 Note that from the commutativity of~\eqref{eq:equiv-corresp}, $\widetilde{W}$ verifies $\widetilde{W}_{gh}=\widetilde{W}_g\circ \widetilde{W}_h$ for all $(g,h)\in \cG^{(2)}$; so that $\widetilde{W}$ is a Rg $\cG$-action by automorphisms on $\cE$. Moreover, it is straightforward to see that the map $\vp:A\To \cL(\cE)$ is $\cG$-equivariant; \emph{i.e.} $\widetilde{W}_g\vp(a)\widetilde{W}_g^\ast=\vp(\al_g(a))$ for all $g\in \cG$ and $a\in A_{s(g)}$. Now, if $(\cE,F)$ is a $\cG$-equivariant Kasparov $A,B$-correspondence, then condition (iv) of Definition~\ref{df:Kas-corresp} implies that $\widetilde{W}_gF_{s(g)}\widetilde{W}_g^\ast-F_{r(g)}\in \cK(\cE_{r(g)})$ for all $g\in \cG$ (take $\xi=0$), so that we recover Le Gall's definition of an equivariant Kasparov bimodule (~\cite[D\'efinition 5.2]{LeGall:KK_groupoid}). We will then refer to condition (iv) as the condition of \emph{invariance modulo compacts}.
\end{rem}

\begin{df}
Two equivariant Kasparov $A,B$-correspondences $(\cE_i,F_i),i=1,2$ are \emph{unitarily equivalent} if there exists an isomorphism of $A,B$-correspondences $u\in \cL(\cE_1,\cE_2)$ such that $F_2=uF_1u^\ast$; in this case we write $(\cE_1,F_1)\sim_u(\cE_2,F_2)$. The set of unitarily equivalence classes of equivariant Kasparov $A,B$-correspondences is denoted by $\EER_\cG(A,B)$.
\end{df}

Let the $\cstar$-algebra $A$ be endowed with a generalised $\cG$-action $\sA$. Then the $\cstar$-algebra $A[0,1]:=C([0,1],A)$ (with the grading $(A[0,1])^i=A^i[0,1],i=0,1$, and "real" structure $\bar{f}(t):=\overline{f(t)}, \ \text{for} \ f\in A[0,1], t\in [0,1]$) is equipped with the generalised $\cG$-action given by the "real" graded u.s.c Fell bundle $\sA[0,1]\To \cG$ with $(\sA[0,1])_g=\sA_g[0,1]$.

\begin{df}
Let $A$ and $B$ be $\cstar$-algebras endowed with generalised $\cG$-actions. A \emph{homotopy} in $\EER_\cG(A,B)$ is an element $(\cE,F)\in \EER_\cG(A,B[0,1])$. Two elements $(\cE_i,F_i), i=0,1$ of $\EER_\cG(A,B)$ are said to be \emph{homotopically equivalent} if there is a homotopy $(\cE,F)$ such that $(\cE\hat{\otimes}_{ev_0}B,F\hat{\otimes}_{ev_0}\Id)\sim_u (\cE_0,F_0)$, and $(\cE\hat{\otimes}_{ev_1}B,F\hat{\otimes}_{ev_1}\Id)\sim_u (\cE_1,F_1)$, where for all $t\in [0,1]$, the evolution map $ev_t:B[0,1] \To B$ is the surjective ${}^\ast$-homomorphism $ev_t(f):=f(t)$. The set of homotopy classes of elements of $\EER_\cG(A,B)$ is denoted by $\KKR_\cG(A,B)$. 
\end{df}

\begin{ex}
Let $A$ be a $\cstar$-algebra equipped with a generalised "real" $\cG$-action. Then there is a canonical element $\id_A \in \KKR_\cG(A,A)$ given by the class of the equivariant Kasparov $A,A$-correspondence $(A,0)$, where $A$ is naturally viewed as a $A,A$-correspondence via the homomorphism $A\To \cL_A(A)=A$ defined by left multiplication by elements of $A$.	
\end{ex}

\begin{df}
Given two elements $x_1,x_2\in \KKR_\cG(A,B)$, their sum is given by $x_1\oplus x_2=(\cE_1\oplus\cE_2,F_1\oplus F_2)$, where $(\cE_i,F_i), i=1,2$ is any representative of $x_i$.	
\end{df}

Let $A,B$, and $D$ be $\cstar$-algebras endowed with the generalised  $\cG$-actions $\sA, \sB,$ and $\sD$, respectively. Then the $\cstar$-algebras $A\hat{\otimes}_{C_0(X)}D\cong C_0(X;\sA\hat{\otimes}_X\sD)$ and $B\hat{\otimes}_{C_0(X)}D\cong C_0(X;\sB\hat{\otimes}_X\sD)$ are provided with the generalised "real" $\cG$-actions given by the "real" graded u.s.c. Fell bundles $\sA\hat{\otimes}_\cG\sD\To \cG$ and $\sB\hat{\otimes}_\cG\sD\To \cG$. Now if ${}_A\cE_B$ is a $\cG$-equivariant $\cstar$-correspondence via the non-degenerate homomorphism $\vp:A\To \cL(\cE)$, then $\cE\hat{\otimes}_A A\hat{\otimes}_{C_0(X)}D$ is a $\cG$-equivariant $A\hat{\otimes}_{C_0(X)}D, B\hat{\otimes}_{C_0(X)}D$-correspondence via the map 
$$\vp\hat{\otimes}\Id_A\hat{\otimes}\Id_D:A\hat{\otimes}_{C_0(X)}D\To \cL_{B\hat{\otimes}_{C_0(X)}D}(\cE\hat{\otimes}_AA\hat{\otimes}_{C_0(X)}D).$$

\begin{df}
Let $A,B$, and $D$ be as above. We define the group homomorphism
\[\tau_D: \KKR_\cG(A,B)\To \KKR_\cG(A\hat{\otimes}_{C_0(X)}D,B\hat{\otimes}_{C_0(X)}D)\]
by setting 
\[\tau_D([\cE,F]):=[(\cE\hat{\otimes}_AA\hat{\otimes}_{C_0(X)}D,F\hat{\otimes}\id_A\hat{\otimes}\id_D)], \ {\rm for\ } [(\cE,F)]\in \KKR_\cG(A,B). \]
	
\end{df}

The following can be proven as in the ordinary $\KK$-theory where no generalised actions are involved (see~\cite[\S4]{Kasparov:Operator_K}).

\begin{pro}
Under the operations of direct sum, $\KKR_\cG(A,B)$ is an abelian group. Moreover, the assignement $(A,B)\mto \KKR_\cG(A,B)$ is a bifunctor, covariant in $B$ and contravariant in $A$, from the category $\Cor_\cG$ to the category $\mathfrak{Ab}$ of abelian groups	
\end{pro}

Note that the inverse of an element $x\in \KKR_\cG(A,B)$ is given by the class of $(-\cE,-F)$, where $(\cE,F)$ is a representative of $x$, $-\cE$ is the "real" graded $A,B$-correspondence given by $\cE$ with the opposite grading (\emph{i.e.} $(-\cE)^i=\cE^{1-i},i=0,1$) and the same "real" structure, and the non-degenerate homomorphism of "real" graded $\cstar$-algebras $-\vp:A\to \cL(-\cE)$ defined by \[-\vp(a):= \begin{pmatrix}0&\id_{\cE^1} \\ \id_{\cE^0} & 0\end{pmatrix}\vp(a), \ \forall a\in A.\]
Also, as in the usual case, degenerate elements are homotopically equivalent to $(0,0)$, so that they represent the zero element of $\KKR_\cG(A,B)$.

\begin{rem}
One recovers Kasparov's $\KKR(A,B)$ of~\cite{Kasparov:Operator_K} by taking the "real" groupoid $\cG$ to be the point. Indeed, in this case we may omit condition (iv) of Definition~\ref{df:Kas-corresp} since, thanks to~\eqref{eq:equiv-corresp}, for the automorphism $W:\cE\To \cE$ is indeed equal to the identity.
\end{rem}

Higher $\KKR_\cG$-groups are defined in an obvious way. Given a $\cstar$-algebra $A$ endowed with a generalised $\cG$-action $\sA\To \cG$, the u.s.c. Fell bundle 
\[\sA\hat{\otimes}\Cl_{p,q}:=\coprod_{g\in \cG}\sA_g\hat{\otimes}\Cl_{p,q}\]
over $\cG$ is a generalised $\cG$-action on the "real" graded $\cstar$-algebra \[A_{p,q}:=A\hat{\otimes}\Cl_{p,q}\cong C_0(X;\sA\hat{\otimes}\Cl_{p,q}).\] 

\begin{df}
Let $A, B$ be $\cstar$-algebras endowed with generalised $\cG$-actions. Then, the higher $\KKR_\cG$-groups $\KKR_{\cG,j}(A,B)$ are defined by 

\[ \KKR_{\cG,j}(A,B)=\KKR_\cG^{-j}(A,B):=\left\{\begin{array}{ll}\KKR_\cG(A_{j,0},B)\cong \KKR_\cG(A,B_{0,j}), & {\rm if\ } j\ge 0\\ \KKR_\cG(A_{-j,0},B)\cong \KKR_\cG(A,B_{0,-j}), & {\rm if\ } j\leq 0 \end{array}\right.\]
\end{df}

Let us outline the construction of the Kasparov product in groupoid-equivariant $\KKR$-theory for generalised actions. To do this, we need the following 

\begin{thm}(cf.~\cite[Theorem 6.9]{Tu-Xu-Laurent-Gengoux:Twisted_K}).
Let $A,D$ and $B$ be separable $\cstar$-algebras endowed with generalised $\cG$-actions. Let $(\cE_1,F_1)\in \EER_\cG(A,D), (\cE_2,F_2)\in \EER_\cG(D,B)$. Denote by $\cE$ the $\cG$-equivariant $A,B$-correspondence $\cE=\cE_1\hat{\otimes}_D\cE_2$. Then the set $F_1\hat{\#}F_2$ of "real" operators $F\in \cL(\cE)$ such that 
\begin{itemize}
\item $(\cE,F)\in \EER_\cG(A,B)$;
\item $F$ is a $F_2$-connection for $\cE_1$;
\item $\forall a\in A, a[F_1\hat{\otimes}_D\id, F]a^\ast\geq 0$ modulo $\cK(\cE)$
 	\end{itemize}
is non-empty. 	
\end{thm}

From this theorem, the \emph{Kasparov product} 

\begin{equation}~\label{eq:KKR-coupling}
\hat{\otimes}_{\cG,D}:\KKR_\cG(A,D)\otimes_D \KKR_\cG(D,B) \To \KKR_\cG(A,B)
\end{equation}

of $[(\cE_1,F_1)]\in \KKR_\cG(A,D)$ and $[(\cE_2,F_2)]\in \KKR_\cG(D,B)$ is defined by 
\begin{equation}
[(\cE_1,F_1)]\hat{\otimes}_D[(\cE_2,F_2)]:=[(\cE,F)],
\end{equation}
where $\cE:=\cE_1\hat{\otimes}\cE_2$ and $F\in F_1\hat{\#}F_2$. It is not hard to see that as in the complex case where the $\cstar$-algebras are equipped with $\cG$-actions by automorphisms (see~\cite{LeGall:KK_groupoid}), this product is well-defined, bilinear, associative, homotopy-invariant, covariant with respect to $B$ and contravariant with respect to $A$.\\
More generally, we have 

\begin{thm}
Let $A_1, A_2, B_1, B_2$ and $D$ be separable $\cstar$-algebras endowed with generalised  $\cG$-actions. Then, the product~\eqref{eq:KKR-coupling} induces an associative product 
\[\KKR_{\cG,i}(A_1,B_1\hat{\otimes}_{C_0(X)}D)\otimes_D \KKR_{\cG,j}(D\hat{\otimes}_{C_0(X)}A_2,B_2)\To \KKR_{\cG,i+j}(A_1\hat{\otimes}_{C_0(X)}A_2,B_1\hat{\otimes}_{C_0(X)}B_2).\]	
\end{thm}

\begin{proof}
The proof is almost the same as that of~\cite[Theorem 5.6]{Kasparov:Operator_K}.
\end{proof}

Moreover, as in the usual case (cf.~\cite{Kasparov:Novikov}), there are \emph{descent morphisms} 
\begin{eqnarray*}
j_{\cG}: \KKR_\cG(A,B) & \To & \KKR(\cstar(\cG;\sA),\cstar_r(\cG;\sB));\\
j_{\cG,red}: \KKR_\cG(A,B) & \To & \KKR(\cstar_r(\cG,\sA),\cstar_r(\cG,\sB)),
\end{eqnarray*}
compatible with the Kasparov product.


\section{Functoriality in the Hilsum-Skandalis category}

Recall~\cite{Moutuou:Real.Cohomology} that a ("real") generalised morphism (or a Hilsum-Skandalis morphism) from a "real" groupoid $\Ga$ to the "real" groupoid $\cG$ consists of a "real" space $Z \ni z\mto \bar{z}\in Z$, two continuous maps $\xymatrix{Y & Z\ar[l]_\fr \ar[r]^\fs & X}$ which are equivariant with respect to the "real" structures, a continuous left (resp. right) action of $\Ga$ (resp. of $\cG$) on $Z$ respecting the involutions (\emph{i.e.} $\overline{\g\cdot z\cdot g}=\tau_\Ga(g)\cdot \bar{z}\cdot \tau_\cG(g)$ where the product makes sense), making $Z$ a generalised morphism in the usual sense (see for instance~\cite{Hilsum-Skandalis:Morphismes, Tu-Xu-Laurent-Gengoux:Twisted_K}). There is a category $\RG$ whose objects are "real" groupoids and whose morphisms are (Morita equivalence classes of) "real" generalised morphisms. This category was shown in~\cite[Proposition 1.37]{Moutuou:Real.Cohomology} to be isomorphic to a category $\RG_\Omega$, which is much more exploitable. The category $\RG_\Om$ has same objects as $\RG$. If $\Ga, \cG$ are "real" groupoids, $\Hom_{\RG_\Om}(\Ga,\cG)$ consists of Morita equivalence classes of compositions of the form 
 \[
 \xymatrix{\Ga & \Ga[\cU] \ar@{_{(}->}[l]_\iota \ar[r]^f & \cG}
 \]
 where $\cU=(U_i)_{i\in I}$ is a "real" open cover of $Y$ (that is, $I$ has involution $i\mto \bar{i}$ satisfying $U_{\bar{i}}=\tau(U_i), \forall i\in I$), and $f$ is a strict "real" morphism; \emph{i.e.} equivariant with respect to the involution $(i,\g,j)\mto (\bar{i},\tau_\Ga(\g),\bar{j})$ on the cover groupoid $\Ga[\cU]$ (induced from the those of $\Ga$ and $I$) and that of $\cG$ Such a morphism will be represented by the couple $(\cU,f)$. \\

In this section we show that $\KKR_{(\cdot)}$ is functorial in the category of locally compact second-countable "real" groupoids and generalised "real" morphisms. We first need to show that it is functorial with respect to strict "real" morphisms.

Let $f:\Ga\To \cG$ be a strict morphism of "real" groupoids and $A$ a $\cstar$-algebra endowed with the generalised "real" $\cG$-action $\sA\To \cG$. Then the pull-back $f^\ast\sA\To \Ga$ defines a generalised "real" $\Ga$-action on the $\cstar$-algebra $f^\ast A=C_0(Y;f^\ast\sA)$. Let $B$ be another $\cstar$-algebra together with a generalised "real" $\cG$-action $\sB$. Suppose ${}_A\cE_B$ is a $\cstar$-correspondence. Then under the identifications 

\begin{equation}~\label{eq:KKR_identifications1}
f^\ast A=A\hat{\otimes}_{C_0(X)}C_0(Y), f^\ast B=B\hat{\otimes}_{C_0(X)}C_0(Y), \ {\rm and\ } f^\ast \cE=\cE\hat{\otimes}_{C_0(X)}C_0(Y),
\end{equation}

we see that $f^\ast \cE$ is a "real" graded $f^\ast A,f^\ast B$-correspondence. Further, assume that ${}_A\cE_B$ is $\cG$-equivariant with respect to the isomorphism $W:s_\cG^\ast\cE\hat{\otimes}_{s_\cG^\ast B}\cF(i_\cG^\ast\sB)\stackrel{\cong}{\To} \cF_b(i_\cG^\ast\sA)\hat{\otimes}_{r_\cG^\ast A}r_\cG^\ast\cE$. Then, by using the canonical identifications 
\begin{equation}
\begin{array}{ll} 
i_\Ga^\ast(f^\ast\sA)=f^\ast(i^\ast_\cG\sA), & i_\Ga^\ast(f^\ast\sB)=f^\ast(i_\cG^\ast\sB),\\
s^\ast_\Ga(f^\ast A)=s_\cG^\ast A\hat{\otimes}_{C_0(\cG)}C_0(\Ga), & r_\Ga^\ast(f^\ast A)= r_\cG^\ast A\hat{\otimes}_{C_0(\cG)}C_0(\Ga)\\
s^\ast_\Ga(f^\ast\cE) = r_\cG^\ast\cE\hat{\otimes}_{C_0(\cG)}C_0(\Ga), & r^\ast_\Ga (f^\ast\cE)= r_\cG^\ast\cE\hat{\otimes}_{C_0(\cG)}C_0(\Ga),	
\end{array}
\end{equation}
where the "real" action of $C_0(\cG)$ on $C_0(\Ga)$ is induced by $f$ in an obvious way, we get an isomorphism $$f^\ast W:s^\ast_\Ga(f^\ast\cE)\hat{\otimes}_{s_\Ga^\ast f^\ast B}\cF_b(i_\Ga^\ast f^\ast\sB)\To \cF_b(i_\Ga^\ast f^\ast \cA)\hat{\otimes}_{r_\Ga^\ast f^\ast A}r_\Ga^\ast f^\ast\cE,$$
making $f^\ast\cE$ into a $\cG$-equivariant $f^\ast A,f^\ast B$-correspondence. Hence equivariant $\KKR$-theory has a functorial property in the category $\RG_s$.  

\begin{dflem}
Let $f:\Ga\To \cG$ be a strict morphism of "real" groupoids. Let $A$ and $B$ be $\cstar$-algebras endowed with generalised "real" $\cG$-actions. Then we define a group homomorphism

\[
f^\ast: \KKR_\cG(A,B)\To \KKR_\Ga(f^\ast A,f^\ast B)
\]
by assigning to a $\cG$-equivariant Kasparov $A,B$-correspondence $(\cE,F)\in \EER_\cG(A,B)$ the pair 
\[f^\ast(\cE,F):=(f^\ast \cE,f^\ast F)\in \EER_\Ga(f^\ast A,f^\ast B),\]
where under the identifications~\eqref{eq:KKR_identifications1}, we put $f^\ast F=F\hat{\otimes}_{C_0(X)}\Id_{C_0(Y)}$. \\
Moreover, the map $f^\ast$ is natural with respect to the Kasparov product~\eqref{eq:KKR-coupling} in the sense that if $D$ is another $\cstar$-algebra equipped with a generalized "real" $\cG$-action, then
\[f^\ast(x)\hat{\otimes}_{f^\ast D}f^\ast(y)=f^\ast(x\hat{\otimes}_Dy), \ \forall x\in \KKR_\cG(A,D), y\in \KKR_\cG(D,B).\] 
\end{dflem}

\begin{proof}
The proof is the same as that of~\cite[Proposition 7.2]{LeGall:KK_groupoid}.	
\end{proof}

More generally, in order to establish functoriality in the category $\RG$ we need the following result.

\begin{pro}~\label{pro:pushforward-Fell-bdle}
Let $\grpd$ be a locally compact second-countable "real" groupoid, and let $\cU=(U_j)_{j\in J}$ a "real" open cover of $X$. Denote by $\iota:\cG[\cU]\To \cG$ the canonical inclusion. For all Fell bundle (resp. u.s.c. Fell bundle) $\sA\To \cG[\cU]$, there exist a  Fell bundle (resp. u.s.c. Fell bundle) $\iota_{!}\sA\To \cG$ and an isomorphism of Fell bundles (resp. of u.s.c. Fell bundles) $$\iota^\ast(\iota_!\sA)\stackrel{\cong}{\To}\sA$$ over $\xymatrix{\cG[\cU]\dar[r] & \coprod_{j\in J}U_j}$.
\end{pro}
 
\begin{proof}
Without loss of generality, we may suppose $\cU$ is locally finite since $X$ is paracompact. We use the following notations: as usual, we write $g_{j_0j_1}$ for $(j_0,g,j_1)\in \cG[\cU]$, and $x_j$ for $(j,x)\in \coprod_{j\in J}U_j$; let $\pi:\sA\To \cG[\cU]$ be the projection of the given "real" graded Fell bundle (resp. u.s.c. Fell bundle); an element $a\in \pi^{-1}(g_{j_0j_1})$ will be written $a_{j_0j_1}$. 

For $x\in X$ we denote by $I_x$ the finite subset of $j\in J$ such that $x\in U_j$, and for $g\in \cG$ let $I_g$ the finite subset of pairs $(j_0,j_1)\in J\times J$ such that $g\in \cG^{U_{j_0}}_{U_{j_1}}$. Put
\[\iota_!\sA_g:= \bigoplus_{(j_0,j_1)\in I_g}\sA_{g_{j_0j_1}}, \quad {\rm and\ } \widetilde{\iota_!\sA}:= \coprod_{g\in \cG}\iota_!\sA.\]

Then $\widetilde{\iota_!\sA}\To \cG$ is a "real" graded Banach bundle (resp. u.s.c. Banach bundle), with the projection $\iota_!\pi$ defined by 
\[\iota_!\pi((g,a_{j_0j_1})_{(j_0,j_1)\in I_g}):=g.\]  

Moreover, for $(g_1,g_2)\in \cG^{(2)}$, the pairing 
\begin{equation*}
\begin{array}{lll}
	I_{g_1}\times_{I_{s(g_1)}}I_{g_2} & \To & I_{g_1g_2}\\
	((j_0,j_1), (j_1,j_2) )& \mto & (j_0,j_2)
	\end{array}	
	\end{equation*}
where $I_{g_1}\times_{I_{s(g_1)}}I_{g_2}:=\{((j_0,j_1),(j_1,j_2))\in I_{g_1}\times I_{g_2}\}$, enables us to define a multiplication on $\widetilde{\iota_!\sA}$ 
\[\iota_!\sA_{g_1}\hat{\otimes}_{\iota_!\sA_{s(g_1)}}\iota_!\sA_{g_2} \To \iota_!\sA_{g_1g_2}\]
generated by $a_{j_0j_1}\hat{\otimes}b_{j_1j_2}\mto (ab)_{j_0j_2}$. One easily verifies that together with this multiplication, $\iota_!\sA \To \cG$ is a "real" graded Fell bundle (resp. u.s.c Fell bundle) that satisfies tthe desired isomorphism.	
\end{proof}

\begin{df}
Let $\grpd$ and $\cU$ be as above. Given a $\cstar$-algebra $A$ with a generalised "real" $\cG[\cU]$-action $\sA$, we denote by $\iota_!A$ the $\cstar$-algebra $C_0(X;\iota_!\sA)$ endowed with the generalised "real" $\cG$-action $\iota_!\sA\To \cG$. 	
\end{df}

\begin{pro}
Let $A,B$ be $\cstar$-algebras endowed with generalised "real" $\cG[\cU]$-actions $\sA$ and $\sB$, respectively. Assume $\cE$ is a $\cG[\cU]$-equivariant  $A,B$-correspondence. Then there exists a $\iota_!A,\iota_!B$-correspondence $\iota_!\cE$ and an isomorphism of $A,B$-correspondences $\iota^\ast \iota_!\cE\cong \cE$.	
\end{pro}

\begin{proof}
The map $\iota_!:B\To \iota_!B$ sending $\phi\in C_0(\coprod_jU_j;\sB)$ to the function $\iota_!\phi \in C_0(X;\iota_!\sB)$ given by $$\iota_!\phi(x):=(\phi((j,x)))_{j\in I_x},$$
is a surjective "real" graded ${}^\ast$-homomorphism. We then can define the push-out $\iota_!\cE$ of the Hilbert $B$-module $\cE$ via $\iota_!$. Let us recall~\cite[\S1.2.2.]{Jensen-Thomsen:Elements_KK} the definition of the "real" graded Hilbert $\iota_!B$-module $\iota_!\cE$. Let $N_{\iota_!}=\left\{\xi\in \cE \mid \iota_!(\<\xi,\xi\>_B)=0\right\}$; denote by $\dot{\xi}$ the image of $\xi\in \cE$ in the quotient space $\cE/_{N_{\iota_!}}$, the latter being a "real" graded $\iota_!B$-module by setting 
\[\dot{\xi}\cdot\iota_!(b):=\overset{.}{\wideparen{\xi b}}, \quad \dot{\xi}\in \cE/_{N_{\iota_!}}, \iota_!(b)\in \iota_!B .\] 
Then define the Hilbert $\iota_!B$-module $\iota_!\cE$ as the completion of $\cE/_{N_{\iota_!}}$ with respect to the "real" graded $\iota_!B$-valued pre-inner product  
\[\<\dot{\xi},\dot{\eta}\>_{\iota_!B}:=\iota_!(\<\xi,\eta\>), \quad \xi,\eta\in \cE.\] 

For $T\in \cL(\cE)$, let $\iota_!T$ be the unique operator in $\cL(\iota_!\cE)$ making the following diagram commute
\[\xymatrix{\cE \ar[r] \ar[d]^T & \iota_!\cE \ar[d]^{\iota_!T} \\ \cE \ar[r] & \iota_!\cE}\]
where the horizontal arrows are the quotient map; \emph{i.e.}, 
\begin{equation}~\label{eq:def-iota_!T}
\iota_!T(\dot{\xi})=\overset{.}{\wideparen{T(\xi)}}, \quad \dot{\xi}\in \iota_!\cE.
\end{equation} 
Hence, the map $\vp:A\To \cL(\cE)$ implementing the $A,B$-correspondence gives rise to a non-degenerate ${}^\ast$-homomorphism $\iota_!\vp: \iota_!A\To \cL(\iota_!\cE)$ such that 
\begin{equation}~\label{eq:def-iota_!phi}
(\iota_!\vp)(\iota_!(a)):=\iota_!(\vp(a)), \quad \forall a\in A. 
\end{equation}
Therefore, $\iota_!\cE$ is a $\iota_!A,\iota_!B$-correspondence. It is not hard to check that $\iota_!\cE$ is isomorphic to $C_0(X;\widetilde{\iota_!\sE})$, where $\widetilde{\iota_!\sE}\To X$ is the unique u.s.c. field of Banach algebras with fibre $(\iota_!\sE)_x=\bigoplus_{j\in I_x}\cE\hat{\otimes}_BB_{(j,x)}$; indeed, since $\cE$ is a "real" graded Hilbert $B$-module, thanks to~\cite[Appendix A]{Tu-Xu-Laurent-Gengoux:Twisted_K}) there is a unique topology on the u.s.c. field $\tilde{\sE}=\coprod_{(j,x)} \cE\hat{\otimes}_BB_{(j,x)}$ such that $\cE\cong C_0(\coprod_jU_j;\tilde{\sE})$ the $\ZZ_2$-grading and the "real" structure on $\tilde{\sE}$ is the obvious one; so, by Proposition~\ref{pro:pushforward-Fell-bdle}, we get the "real" graded u.s.c. field $\widetilde{\iota_\sE}$. 

Let $W:s^\ast\cE\hat{\otimes}_{s^\ast B}\cF_b(i^\ast\sB)\To \cF_b(i^\ast\sA)\hat{\otimes}_{r^\ast A}r^\ast\cE$ be the isomorphism of $\cstar$-correspondences implementing the $\cG[\cU]$-equivariance. Then from the identifications 
\begin{eqnarray*}
 		s^\ast (\iota_!A)=\iota_!(s^\ast A), &  r^\ast(\iota_!A)=\iota_!(r^\ast A), \\
 		s^\ast(\iota_!B)=\iota_!(s^\ast B), & r^\ast(\iota_!B)=\iota_!(r^\ast B), \\
 		i^\ast(\iota_!\sA) = \iota_!(i^\ast\sA), & i^\ast(\iota_!\sB)=\iota_!(i^\ast\sB),\\
 		s^\ast(\iota_!\cE)=\iota_!(s^\ast\cE), & r^\ast(\iota_!\cE)=\iota_!(r^\ast\cE),  		
 	\end{eqnarray*}
 	
we get 
\begin{eqnarray*}
s^\ast(\iota_!\cE)\hat{\otimes}_{s^\ast(\iota_!B)}\cF_b(i^\ast(\iota_!\sB)) & \cong & \iota_!\left(s^\ast\cE\hat{\otimes}_{s^\ast B}\cF_b(i^\ast\sB)\right), \ {\rm and\ } \\
\cF_b(i^\ast(\iota_!\sA))\hat{\otimes}_{r^\ast(\iota_!A)}r^\ast(\iota_!\cE) & \cong & \iota_!\left(\cF_b(i^\ast\sA)\hat{\otimes}_{r^\ast A}r^\ast\cE\right).
\end{eqnarray*} 
Thus, $W$ induces an isomorphism of $s^\ast(\iota_!A),r^\ast(\iota_!B)$-correspondences
 \[\iota_!W: s^\ast(\iota_!\cE)\hat{\otimes}_{s^\ast(\iota_!B)}\cF_b(i^\ast(\iota_!\sB))\To \cF_b(i^\ast(\iota_!\sA))\hat{\otimes}_{r^\ast(\iota_!A)}r^\ast(\iota_!\cE),
 \] 
 defined in a similar fashion as~\eqref{eq:def-iota_!phi}; so that $\iota_!W$ is compatible with the partial product of $\cG$ in the sense of the commutative diagram~\eqref{eq:equiv-corresp}. That the $\cG[\cU]$-equivariant $A,B$-correspondences $\cE$ and $\iota^\ast\iota_!\cE$ are isomorphic is an immediate consequence of the construction of $\iota_!\cE$ and $\iota_!W$. 		
\end{proof}

\begin{lem}
Suppose $x=(\cE,F)\in \EER_{\cG[\cU]}(A,B)$. Then $\iota_!x:=(\iota_!\cE,\iota_!F)\in \EER_\cG(\iota_!A,\iota_!B)$, where the operator $\iota_!F\in \cL(\iota_!\cE)$ is given by~\eqref{eq:def-iota_!T}.	
\end{lem}

\begin{proof}
This is a matter of algebraic verifications. For instance, the map $\iota_!:\cL(\cE)\To \cL(\iota_!\cE)$ respects the degree and the "real" structures. Moreover $\iota_!$ sends $\cK(\cE)$ onto $\cK(\iota_!\cE)$ because $\iota_!(\theta_{\xi,\eta})=\theta_{\dot{\xi},\dot{\eta}}$ for all $\xi,\eta\in \cE$, $(\iota_!T_1)(\iota_!T_1)=\iota_!(T_1T_2)$, and $[\iota_!T_1,\iota_!T_2]=\iota_![T_1,T_2], , \forall T_1,T_2\in \cL(\cE)$; thus conditions (i)-(iii) in Definition~\ref{df:Kas-corresp} are satisfied by the pair $(\iota_!\cE,\iota_!F)$. To verify condition (iv), let us put $\cE_1=\cF_b(i^\ast\sA), \cE_2=r^\ast\cE$, and $\cE=\cE_1\hat{\otimes}_{r^\ast A}\cE_2$. Then $\cK(\iota_!\cE_2\oplus \iota_!\cE)=\cK(\iota_!(\cE_1\oplus\cE))$, and we have seen in the proof of Proposition~\ref{pro:pushforward-Fell-bdle} that $\iota\cE_1\hat{\otimes}_{\iota_!r^\ast A}\iota_!\cE=\iota_!(\cE_1\hat{\otimes}_{r^\ast A}\cE)$. It follows that for $\xi\in \cE_1$ and $\eta\in\cE_2$ , one has 
\[
\iota_!(T_\xi)=\overset{.}{\wideparen{T_\xi(\eta)}}=\overset{.}{\wideparen{\xi \otimes_{r^\ast A}\eta}}=\dot{\xi}\hat{\otimes}_{_{\iota_!r^\ast A}}\dot{\eta},
\]
which means that $\iota_!(T_\xi)=T_{\dot{\xi}}\in \cL(\iota_!\cE_2,\iota_!\cE)$, and hence $\iota_!(\theta_{\xi})=\theta_{\dot{\xi}}$ (recall notations used in Remark~\ref{rem:F-connection}). Then, 
\begin{eqnarray*}
[\theta_{\dot{\xi}}, \iota_!(r^\ast F)\oplus \iota_!(W(s^\ast F\hat{\otimes}_{s^\ast B \Id})W^\ast)] & = & \iota_!([\theta_\xi, r^\ast F\oplus W(s^\ast F\hat{\otimes}_{s^\ast B \Id})W^\ast] \\
 & \in & \cK(\iota_!(\cE_2\oplus \cE)) =\cK(\iota_!\cE_2\oplus \iota_!\cE);
\end{eqnarray*}
therefore, $\iota_!W(s^\ast(\iota_!F)\hat{\otimes}_{s^\ast\iota_!B}\Id)\iota_!W^\ast= \iota_!(W(s^\ast F\hat{\otimes}_{s^\ast B \Id})W^\ast)$ is an $r^\ast \iota_!F$-connection for $\iota_!\cE_1=\cF_b(i^\ast (\iota_!\sA))$.
\end{proof}

The following result can be proved with similar arguments as in~\cite[Th\'eor\`eme 7.1]{LeGall:KK_groupoid}, so we omit the proof.

\begin{thm}~\label{thm:functor-iota_!-in-KKR_G}
Let $\grpd$, $\cU$, $A$ and $B$ be as previously. Then the canonical "real" inclusion $\cG[\cU]\hookrightarrow \cG$ induces a group isomorphism
\[\iota^\ast: \KKR_{\cG}(\iota_!A,\iota_!B)\To \KKR_{\cG[\cU]}(A,B),\]
whose inverse is 
\[\iota_!: \KKR_{\cG[\cU]}(A,B)\ni [(\cE,F)]\mto [(\iota_!\cE,\iota_!F)]\in \KKR_{\cG}(\iota_!A,\iota_!B).\]	
\end{thm}

\begin{cor}
The isomorphism $\iota_!: \KKR_{\cG[\cU]}(A,B)\To \KKR_\cG(\iota_!A,\iota_!B)$ is natural with respect to Kasparov product; \emph{i.e.}, 
\[\iota_!x\hat{\otimes}_{\iota_!D}\iota_!y=\iota_!(x\hat{\otimes}_Dy), \forall x\in \KKR_{\cG[\cU]}(A,D), y\in \KKR{\cG[\cU]}(D,B).\]	
\end{cor}

Now Theorem~\ref{thm:functor-iota_!-in-KKR_G} enables us to define the pull-back of a  $\cstar$-algebra endowed with a generalised action along a morphism in the category  $\RG$. 

\begin{df}
Let $Z:\Ga\To \cG$ be a generalised "real" homomorphism, $A$ and $B$ be $\cstar$-algebras endowed with generalised "real" $\cG$-actions. Let $(\cU,f)$ be a representative of a morphism in $\RG_\Om$ realisng $Z$. Let $\iota:\Ga[\cU]\To \Ga$ be the canonical inclusion. Define the pull-back of $A$ and $B$ along $Z$ by
\begin{equation*}
Z^\ast B:=\iota_!f^\ast A, \quad {\rm and\ } Z^\ast B:=\iota_!f^\ast B.
\end{equation*}
Then we define the pull-back homomorphism 
\[Z^\ast: \KKR_\cG(A,B) \To \KKR_\Ga(Z^\ast A,Z^\ast B),\]
as the composite 
\[\KKR_\cG(A,B)\stackrel{f^\ast}{\To} \KKR_{\Ga[\cU]}(f^\ast A,f^\ast B) \stackrel{\iota_!}{\To} \KKR_\Ga(\iota_!f^\ast A,\iota_!f^\ast B),\] 	
\end{df}

Notice that by similar arguments as in~\cite[Th\'eor\`eme 7.2]{LeGall:KK_groupoid}, $Z^\ast$ is well defined, that is, its construction does not depend on the choice of the pair $[(\cU,f)]$, and that it is functorial in $\RG$ and is natural with respect to Kasparov product. In particular, if $Z$ is a Morita equivalence, then $Z^\ast$ is a group isomorphism.


\section{$\KKR_\cG$-equivalence}

The notion of $\KK$-equivalence, which has already been treated in many  textbooks and papers,  provides important features in the study of $\K$-amenability and the Baum-Connes conjecture (see Blackadar's book~\cite[Definition 19.1.1]{Blackadar:K-theory1}). It also gives a powerful way to establish Bott periodicity in $\K$-theory of $\cstar$-algebras. 

This notion as well as the results around it extend in a very natural way to the more general setting of equivariant $\KK$-theory for generalised "real" groupoid actions.  

\begin{df}.
Let $\grpd$, be a "real" groupoid, $A$ and $B$ be $\cstar$-algebras endowed with generalised $\cG$-actions. We say that an element $x\in \KKR_\cG(A,B)$ is a $KKR_\cG$-\emph{equivalence} if there is $y\in \KKR_\cG(B,A)$ such that 
\[x\hat{\otimes}_{\cG,B}y=\id_A, \quad {\rm and\ } y\hat{\otimes}_{\cG,A}x=\id_B.\]
$A$ and $B$ are said \emph{$\KKR_\cG$-equivalent} if there exists a $\KKR_\cG$-equivalence in $\KKR_\cG(A,B)$.	
\end{df}

\begin{lem}~\label{lem:KKR_G-equivalence}
Assume $x\in \KKR_\cG(A,B)$ is a $\KKR_\cG$-equivalence. Then for any "real" graded $\cstar$-algebra $D$ endowed with a generalised "real" $\cG$-action, the maps 
\[x\hat{\otimes}_{\cG,B} (\cdot): \KKR_\cG(B,D)\To \KKR_\cG(A,D), \quad {\rm and\ } (\cdot)\hat{\otimes}_{\cG,A}x:\KKR_\cG(D,A)\To \KKR_\cG(D,B),\]
are isomorphisms which are natural in $D$ by associativity.	
\end{lem}

The proof is almost the same as that of~\cite[Theorem 4.6]{Kasparov:Operator_K}. For instance, the map 
$$\KKR_\cG(A,D)\ni z\mto y\hat{\otimes}_{\cG,A}z\in \KKR_\cG(B,D)$$ 
is an inverse of the first homomorphism. See also~\cite[\S19.1]{Blackadar:K-theory1}).

\begin{pro}~\label{pro:functorial-KKR_G-equivalence}
$\KKR_\cG$-equivalence is functorial in $\RG$ in the following sense: if $Z:\Ga\To \cG$ is a generalised "real" homomorphism, and if $A$ and $B$ are $\KKR_\cG$-equivalent, then $Z^\ast A$ and $Z^\ast B$ are $\KKR_\Ga$-equivalent.	
\end{pro}

\begin{proof}
By naturality of $Z^\ast$ with respect to Kasparov product, we have 
\[\id_{Z^\ast A}=Z^\ast(x\hat{\otimes}_{\cG,A}y)=Z^\ast(x)\hat{\otimes}_{\Ga,Z^\ast A}Z^\ast(y),\]
and 
\[\id_{Z^\ast B}=Z^\ast(y\hat{\otimes}_{\cG,B}x)=Z^\ast(y)\hat{\otimes}_{\Ga,Z^\ast B}Z^\ast(x);\]
therefore $Z^\ast(x)\in \KKR_\Ga(Z^\ast A,Z^\ast B)$ is a $\KKR_\Ga$-equivalence.	
\end{proof}


\section{Bott periodicity}

In this section we establish Bott periodicity in $\KKR_\cG$-theory. We first need some definitions and constructions.

\begin{df}
Let $\grpd$ be a "real" groupoid. A \emph{"real" Euclidean vector bundle of type $p-q$} over $\grpd$ is a Euclidean vector bundle $\pi:E\To X$ of rank $p+q$ equipped with a "real" $\cG$-action (with respect to $\pi$) such that the Euclidean metric is $\cG$-invariant and the "real" space $E$ is locally homeomorphic to $\RR^{p,q}$; that is to say, for every $x\in X$, there is a "real" open neigborhood $U$ of $x$ and a "real" homeomorphism $h_U:\pi^{-1}(U)\To U\times \RR^{p,q}$, where $U\times \RR^{p,q}$ is provided with the "real" structure $(x,t)\mto (\bar{x},\bar{t})$. This is equivalent to the existence of a "real" open cover $\cU=(U_j)_{j\in J}$ and a family  of homeomorphisms $h_j:\pi^{-1}(U_j)\To U_j\times \RR^{p+q}$ such that the following diagram commute
\begin{equation}
		\xymatrix{\pi^{-1}(U_j) \ar[r]^{h_j} \ar[d]_{\tau} & U_j\times \RR^{p+q} \ar[d]^{\tau\times (\id_p-\id_q)}\\ \pi^{-1}(U_{\bar{j}})\ar[r]^{h_{\bar{j}}} & U_{\bar{j}}\times \RR^{p+q}}
	\end{equation}	
\end{df}

For $p,q\in \NN$ with $n=p+q\neq 0$, we define the "real" group $O(p+q)$ to be the orthogonal group $O(n)$ equipped with the involution induced from $\RR^{p,q}$ (we identify $M_{p+q}(\RR)$ with $\cL(\RR^{p,q})$, the latter is then a "real" space). Similarly one defines the "real" group $SO(p+q)$

\begin{df}
Associated to any "real" Euclidean vector bundle $E$ of type $p-q$ over the "real" groupoid $\grpd$, there is a generalised "real" homomorphism 
\[\FF(E):\cG\To O(p+q),\]
where $\FF(E)$ is the frame bundle of $E\To X$.	
\end{df}

\begin{rem}
The above definition does make sense, for the fibre of the $O(p+q)$-principal bundle $\FF(E)\To X$ at a point $x\in X$ identifies to the $\RR$-linear space $\Isom(\RR^{p+q},E_x)$ of $\RR$-linear isomorphisms; so that $\cG$ acts on $\FF(E)$ by $g\cdot (s(g),\vp)\mto (r(g),g\cdot \vp)$, for $\vp\in \Isom(\RR^{p+q},E_{s(g)})$, where $(g\cdot \vp)(t):=g\cdot \vp(t),t\in \RR^{p+q}$. $\FF(E)$ is equipped with the "real" structure $(x,\vp)\mto (\bar{x},\bar{\vp})$, where $\bar{\vp}(t):=\overline{\vp(\bar{t})},t\in \RR^{p+q}$. It is clear that the actions by $\cG$ and $O(p+q)$ are compatible with this involution. 
\end{rem}

\begin{ex}
The trivial bundle $=X\times \RR^{p,q}\To X$ is a "real" Euclidean vector bundle of type $p-q$ over $\grpd$ with respect to the "real" $\cG$-action $g\cdot (s(g),t)\mto (r(g),t)$. The associated generalised "real" homomorphism $\FF(X\times \RR^{p,q})$ from $\grpd$ to $\xymatrix{O(p+q)\dar[r]& \cdot}$ is isomorphic to the trivial "real" $O(p+q)$-principal $\cG$-bundle $X\times O(p+q)\To X$; we denote it by $\FF_{p,q}$.
\end{ex}

Recall that $\Cl_{p,q}:= \Cl(\RR^{p,q}):=Cl_{p,q}\otimes_{\RR}\CC=Cl(\RR^{p,q})\otimes_{\RR}\CC$ is the "real" graded Clifford $\cstar$-algebra, where the "real" structure is $x\otimes_\RR \lambda \mto \bar{x}\otimes_\RR \bar{\lambda}$, and where the involution "bar" in $Cl(\RR^{p,q})$ is induced from that of $\RR^{p,q}$. The "real" action of $O(p+q)$ on $\RR^{p,q}$ induces a "real" $O(p+q)$-action on $\Cl_{p,q}$. 
 
Recall that Kasparov has defined in~\cite[\S5]{Kasparov:Operator_K} a $\KKR_{O(p+q)}$-equivalence $$\al_{p,q}\in \KKR_{O(p+q)}\left(C_0(\RR^{p,q}),\Cl_{p,q}\right), \quad {\rm and\ }\beta_{p,q}\in \KKR_{O(p+q)}\left(\Cl_{p,q},C_0(\RR^{p,q})\right),$$ 
providing a $\KKR_{O(p+q)}$-equivalence between $C_0(\RR^{p,q})$ and $\Cl_{p,q}$.

We will use these elements to prove \emph{Bott periodicity} in generalised "real" groupoid equivariant $\KK$-theory as well as the \emph{Thom isomorphism} in twisted $\K$-theory which will be discussed in the next section.

\begin{thm}[Bott periodicity]
Let $\grpd$ be a locally compact second-countable "real" groupoid, and let $A$ and $B$ be $\cstar$-algebras endowed with generalised $\cG$-actions. Then the Kasparov product with $\FF_{p,q}^\ast\al_{p,q}\in \KKR_{\cG}\left(\Co(X)\otimes C_0(\RR^{p,q}),C_0(X)_{p,q}\right)$ defines an isomorphism
\begin{eqnarray*}
\KKR_{\cG,i+p-q}(A,B) & \cong & \KKR_{\cG,i}(A(\RR^{p,q}),B),
\end{eqnarray*}
where $A(\RR^{p,q})=C_0(\RR^{p,q};A)=C_0(\RR^{p,q})\otimes A$.	
\end{thm}

\begin{proof}
First of all notice that the pullbacks $\FF^\ast_{p,q}(C_0(\RR^{p,q}))$ and $\FF_{p,q}^\ast(\Cl_{p,q})$ via the generalised homomorphism $\FF_{p,q}:\cG\To O(p+q)$ are isomorphic to $C_0(X;C_0(\RR^{p,q}))=C_0(X)\otimes C_0(\RR^{p,q})$ and $C_0(X)\otimes\Cl_{p,q}=C_0(X)_{p,q}$, respectively. These are then ("real" graded) $\cstar$-algebras endowed with generalised "real" $\cG$-actions. Since $\al_{p,q}\in \KKR_{O(p+q)}(C_0(\RR^{p,q}),\Cl_{p,q})$ is a $\KKR_{O(p+q)}$-equivalence, its pullback $\FF_{p,q}^\ast\al_{p,q}\in \KKR_\cG(C_0(X)\otimes C_0(\RR^{p,q}),C_0(X)_{p,q})$ is a $\KKR_\cG$-equivalence, thanks to Proposition~\ref{pro:functorial-KKR_G-equivalence}. Hence, from Lemma~\ref{lem:KKR_G-equivalence}, the Kasparov product
\[
\xymatrix @!0 @C=4pc @R=3pc{\KKR_\cG\left(C_0(X)\hat{\otimes}C_0(\RR^{p,q}),C_0(X)\otimes \Cl_{p,q}\right) \otimes_{C_0(X;\Cl_{p,q})}  \KKR_{\cG,i}\left(C_0(X)\otimes_{C_0(X)}\Cl_{p,q}\hat{\otimes}A,B\right) \ar[dd]^{\FF_{p,q}^\ast\al_{p,q}\hat{\otimes}_{_{\cG,C_0(X;\Cl_{p,q})}}(\cdot)} \\ 
\\ 
\KKR_{\cG,i}\left(A(\RR^{p,q}),B\hat{\otimes}_{C_0(X)}C_0(X)\right)} 
\]
is an isomorphism. We therefore have the desired isomorphism since $C_0(X)\otimes_{C_0(X)}\Cl_{p,q}\hat{\otimes}A\cong A\hat{\otimes}\Cl_{p,q}$.
\end{proof}


\section{Twisting by "real" Clifford bundles and Stiefel-Whitney classes}

In this section we use the previous constructions to prove some new results in twisted groupoid $\K$-theory. \\
First recall~\cite{Moutuou:Brauer_Group} that a \emph{"real" graded $\uc$--extension} over the "real" groupoid $\cG$ is a graded extension (see~\cite{Tu:Twisted_Poincare}) $(\wGa,\Ga,Z)$, such that the groupoids $\wGa, \Ga$ are "real", $Z$ is a "real" generalised morphism $Z:\Ga\To \cG$, and if the groups $\uc, \ZZ_2$  are given the involution by complex conjugation and the trivial involution, respectively, all of the maps 
\[
\xymatrix{\uc \ar[r] & \wGa \ar[r]^\pi &\Ga \ar[d]^\del \\ & & \ZZ_2}
\]  
are equivariant. The set $\wRExt(\cG,\uc)$ of Morita equivalence classes o "real" graded $\uc$-entension over $\cG$ has the structure of abelian group. Moreover, by~\cite[Theorem 2.60]{Moutuou:Real.Cohomology}, there is an isomorphism $dd:\wRExt(\cG,\uc)\overset{\cong}{\To} \check{H}R^1(\cG_\bullet,\ZZ_2)\ltimes \check{H}R^2(\cG_\bullet,\uc)$, which is natural in the category $\RG$.  It follows that there is a natural isomorphism $\wRBr_0(\cG)\overset{\cong}{\To} \wRExt(\cG,\uc)$ (note that the construction of this isomorphism is explicitly given in~\cite{Moutuou:Brauer_Group}), where the left hand side is the subgroup of $\wRBr(\cG)$ consisting of "real" graded D-D bundles of type $0$.

Let $n=p+q\in \NN^\ast$. The group $\pin(p+q)$ is defined  as
\[
\pin(p+q):= \left\{\g\in \Cl_{p,q} \mid \ve(\g)v\g^\ast\in \RR^{p,q}, \forall v\in \RR^{p,q}, \ {\rm and\ } \g\g^\ast=1 \right\},
\]
where $\ve$ is the canonical $\ZZ_2$-grading of $\Cl_{p,q}$. It is known (~\cite[\S IV.4]{Karoubi:K-theory_Intro}) that 
\[
\pin(p+q)\cong \{\g=x_1\cdots x_k \in Cl_{p,q} \mid x_i\in \bfS^{p,q}, 1\le k\le 2n \}.
\] 
It follows that $\pin(p+q)$ is a "real" group with respect to the involution induced from $\bfS^{p,q}$; \emph{i.e.}, $\bar{\g}=\bar{x}_1\cdots \bar{x}_k$, for $\g\in \pin(p+q)$. Of course, this involution is equivalent to that induced from $\Cl_{p,q}$. Moreover, the surjective homomorphism $\pi: \pin(p+q)\To O(p+q),\g\mto \pi_\g$, where $\pi_\g(v):=\ve(\g)v\g^\ast$ for $v\in \RR^{p,q}$, is clearly "real". Notice that $\ker \pi =\{\pm1\}=\ZZ_2$. Hence, there is a canonical "real" graded $\ZZ_2$-central extension of the "real" groupoid $\xymatrix{O(p+q)\dar[r]& \cdot}$ 

\[
\xymatrix{\ZZ_2 \ar[r] & \pin(p+q)\ar[r]^{\pi} & O(p+q)\ar[d]^{\del} \\ & & \ZZ_2}
\]
where the homomorphism $\del:O(p+q)\To \ZZ_2$ is the unique map such that $\det A=(-1)^{\del(A)}$ (compare with~\cite[\S2.5]{Tu:Twisted_Poincare}). Let 
\[
\pin^c(p+q):=\pin(p+q)\times_{\{\pm1\}}\uc,
\]
be endowed with the "real" structure $[(\g,\lambda)]\mto [(\bar{\g},\bar{\lambda})]$, where as usual, the "bar" operation in $\uc$ is the complex conjugation. Then the above "real" graded $\ZZ_2$-central extension induces a "real" graded $\uc$-central extension $\cT_{p,q}$ 

\begin{equation}~\label{eq:ext(O(p+q))-tau_p-q}
 	\xymatrix{\uc \ar[r] & \pin^c(p+q) \ar[r]^{\pi} & O(p+q) \ar[d]^{\del} \\ & & \ZZ_2}
 \end{equation} 
 of $\xymatrix{O(p+q)\dar[r]& \cdot}$. 
 
Let $V$ be a "real" Euclidean vector bundle of type $p-q$ over $\grpd$. Then there is a Rg $\uc$-central extension $\EE^V$ obtained by pulling back $\cT_{p,q}$ via the generalized "real" homomorphism $\FF(V):\cG\To O(p+q)$.

\begin{df}
 Let $V$ be a "real" Euclidean vector bundle of type $p-q$ over $\grpd$. We define its \emph{associated "real" graded D-D bundle} as the "real" graded D-D bundle $\cA_V$ of type $0$ over $\grpd$ whose image in $\wRExt(\cG,\uc)$ is $[\EE^V]$ via the isomorphism $\wRBr_0(\cG)\To \wRExt(\cG,\uc)$.	
\end{df}  

\begin{lem}~\label{lem:A_(sym)-vs-pro-A}
Let $V$ and $V'$ be "real" Euclidean vector bundles of type $p-q$ and $p'-q'$, respectively, over $\grpd$. Then the "real" graded D-D bundles $\cA_{V\oplus V'}$ and $\cA_V\hat{\otimes}_X\cA_{V'}$ are Morita equivalent.
\end{lem}

\begin{proof}
Considering the "real" homomorphisms $\pin^c(p+q)\times \pin^c(p'+q')\To \pin^c((p+p')+(q+q'))$ and $O(p+q)\times O(p'+q')\To O((p+p')+(q+q'))$ (cf.~\cite{Kasparov:Operator_K}) and a "real" open cove of $X$ trivialising both $V$ and $V'$ (and hence the direct sum $V\oplus V'$), one easily checks that $(\FF(V)^\ast\cT_{p,q})\hat{\otimes}(\FF(V')^\ast\cT_{p',q'})\sim (\FF(V\oplus V'))^\ast \cT_{p+p',q+q'}$.	
\end{proof}

\begin{lem}~\label{lem:A_id=0}
Let $\id^{p,q}$ be the trivial "real" Euclidean vector bundle $X\times \RR^{p,q}\To X$ of type $p-q$ over $\grpd$, where the "real" $\cG$-action is $g\cdot (s(g),t)=(r(g),t)$. Then $\cA_{\id^{p,q}}=0$ in $\wRBr_0(\cG)$.	
\end{lem}

\begin{proof}
The generalized "real" homomorphism $\FF(\id^{p,q}):\cG\To O(p+q)$ is but the generalised homomorphism induced by the strict "real" homomorphism $\id:\cG\ni g \mto \id \in O(p+q)$. Hence $\EE^{\id^{p,q}}=(\id^{p,q})^\ast \cT_{p,q}$, and since the kernel of the projection $\pin^c(p+q)\To O(p+q)$ is $\uc$, $\EE^{\id^{p,q}}$ is isomorphic to the trivial extension $(\cG\times \uc,\cG,0)$.	
\end{proof}
  
Now, note that the action of $\cG$ on the complexified bundle $V_\CC:=V\otimes_\RR\CC \To X$ induces a "real" $\cG$-action by graded automorphisms on the complex Clifford bundle 
\[\Cl(V):=Cl(V_\CC)\To X,\]
making it a "real" graded D-D bundle of type $q-p \mod 8$ over $\grpd$. 

We are interesting in comparing the "real" graded D-D bundles $\cA_V$ and $\Cl(V)$ by means of their classes in the "real" graded Brauer group~\cite{Moutuou:Brauer_Group}. In particular, we want to show the following.

\begin{thm}~\label{thm:main-thm-vector}
Let $V$ be a "real" Euclidean vector bundle of type $p-q$ over $\grpd$. Then for all $\cA\in \wRBr(\cG)$, we have 
\[
\KR^{\ast}_{\cA+\Cl(V)}(\cG^\bullet) \cong \KR^{\ast +q-p}_{\cA+\cA_V}(\cG^\bullet).
\]	
\end{thm}

We shall mention that this result is proven by J.-L. Tu in the complex case (see~\cite[Proposition 2.5]{Tu:Twisted_Poincare}). However, the approach we will be using here to prove it is very different from that used by Tu. 

Indeed, our proof requires the construction of  \emph{generalised Stiefel-Whitney classes} of a "real" vector bundles over a "real" groupoid. Recall that associated to any real vector bundle $V$ over a locally compact paracompact space $X$, there are cohomology classes $w_i(V)\in H^i(X,\ZZ_2)$ called the \emph{$i^{th}$ Stiefel-Whitney classes} of $V$ (see for instance~\cite[Chap.17 \S2]{Husemoller:Fibre}). For instance $w_1(V)$ is the constraint for $V$ being \emph{oriented}, and $w_2(V)$ is the constraint for $V$ being $\spin^c$ (we will say more about that later). 
 
We have already seen that a "real" Euclidean vector bundle $V$ of type $p-q$ gives rise to a generalised "real" homomorphism $\FF(V):\cG\To O(p+q)$. In fact, "real" Euclidean vector bundles arise this way: given $P:\cG\To O(p+q)$, $V:=P\times_{O(p+q)}\RR^{p,q}\To X$ is a "real" Euclidean vector bundle of type $p-q$. There then is  a bijection between the set $\operatorname{Vect}_{p+q}(\cG)$ of isomorphism classes of "real" Euclidean vector bundles of type $p-q$ and the set $\Hom_{\RG}(\cG,O(p+q))$, and hence with $\check{H}R^1(\cG_\bullet,O(p+q))$, thanks to~\cite[Proposition 2.49]{Moutuou:Real.Cohomology}. 

Let $\frc$ be a "real" $O(p+q)$-valued $1$-cocycle over $\cG$ realizing $\FF(V)$. This can be considered as a "real" family of continuous maps $\frc_{(j_0,j_1)}: U^1_{(j_0,j_1)}\To O(p+q)$ such that 
\begin{equation}~\label{eq:cocycle-cond-frc}
\frc_{(j_0,j_1)}(\g_1)\frc_{(j_1,j_2)}(\g_2)=\frc_{(j_0,j_2)}(\g_1\g_2), \quad (\g_1,\g_2)\in U^2_{(j_0,j_1,j_2)},
\end{equation}
where $\cU=\{U_j\}_{j\in J}$ is a "real" open cover of $X$ (indeed, if $f:\cG[\cU]\To O(p+q)$ is a "real" homomorphism realising $\FF(V)$, then one can take $\frc_{(j_0,j_1)}(g_{(j_0,j_1)}):=f(g_{j_0j_1})$). We may suppose that the simplicial "real" cover $\cU_\bullet$ of $\cG_\bullet$ is "small" enough so that we can pick a "real" family of continuous maps $\tilde{\frc}_{(j_0,j_1)}: U^1_{(j_0,j_1)}\To \pin^c(p+q)$ which are a $\pin^c(p+q)$-lifting of $(\frc_{(j_0,j_1)})$ through the "real" projection $\pi:\pin^c(p+q)\To O(p+q)$; \emph{i.e.}, $\pi(\tilde{\frc}_{(j_0,j_1)}(\g))=\frc_{(j_0,j_1)}(\g), \forall \g\in U^1_{(j_0,j_1)}$. In view of equation~\eqref{eq:cocycle-cond-frc}, we have 
\begin{equation}~\label{eq:cocycle-cond-frc-tilde}
 \tilde{\frc}_{(j_0,j_1)}(\g_1)\tilde{\frc}_{(j_1,j_2)}(\g_2)=\om_{(j_0,j_1,j_2)}(\g_1,\g_2)\tilde{\frc}_{(j_0,j_2)}(\g_1\g_2), \forall (\g_1,\g_2)\in U^2_{(j_0,j_1,j_2)},	
 \end{equation} 
for some $\om_{(j_0,j_1,j)}(\g_1,\g_2)\in \uc$. The elements $\om_{(j_0,j_1,j)}(\g_1,\g_2)$ clearly define a "real" family of continuous functions $\om_{(j_0,j_1,j)}:U^2_{(j_0,j_1,j_2)}\To \uc$ which are easily checked to be an element of $ZR^2_{ss}(\cU_\bullet,\uc)$.

\begin{df}
Let $V$ be a "real" Euclidean vector bundle of type $p-q$ over $\grpd$. Let $\frc$ be the class of \ $\FF(V)$ in $\check{H}R^1(\cG_\bullet,O(p+q))$.
\begin{itemize}
	\item[(a)] The \emph{first generalized Stiefel-Whitney class} $w_1(V)\in \check{H}R^1(\cG_\bullet,\ZZ_2)$ as $w_1(V):=\del \circ \frc$, where $\del:O(p+q)\To \ZZ_2$ is the homomorphism defined in~\eqref{eq:ext(O(p+q))-tau_p-q}.
	\item[(b)] The \emph{second generalized Stiefel-Whitney class} $w_2(V)$ is the class in $\check{H}R^2(\cG_\bullet,\uc)$ of the "real" $2$-cocycle $\om$ uniquely determined by equation~\eqref{eq:cocycle-cond-frc-tilde}.	
	\end{itemize}	
We define $w(V):=\left(w_1(V),w_2(V)\right)\in \check{H}R^1(\cG_\bullet,\ZZ_2)\times \check{H}R^2(\cG_\bullet,\uc)$.	
\end{df}

\begin{rem}
Note that $w_1(V)=0$ implies that \ $\FF(V)$ is actually a "real" $SO(p+q)$-principal bundle over $\grpd$, which means that $V$ is \emph{oriented}. Moreover, the "real" family $(\om_{(j_0,j_1)})$ is nothing but the obstruction for the "real" $O(p+q)$-valued $1$-cocyle $\frc$ to lift to a "real" $\pin^c(p+q)$-valued $1$-cocyle $\tilde{\frc}$; or in other words, it is the obstruction for \ $\FF(V)$ to lift to a "real" $\pin^c(p+q)$-principal bundle over $\grpd$.
\end{rem}

\begin{ex}~\label{ex:Stiefel-Whitney-theta_p-q}
Denote by \ $\theta^{p,q}$ the trivial Euclidean vector bundle $\RR^{p,q}$ over $\xymatrix{O(p+q)\dar[r] & \cdot}$. Then $w_1(\theta^{p,q})=\del$ and $w_2(\theta^{p,q})=c_{p,q}$ is the "real" $\uc$-valued $2$-cocycle corresponding to the "real" graded $\uc$-central extension $\cT_{p,q} \in \wRExt(O(p+q),\uc)$ of~\eqref{eq:ext(O(p+q))-tau_p-q}. Moreover, $w(\theta^{p,q})=dd(\cT_{p,q})$.   	
\end{ex}

\begin{df}
A "real" Euclidean vector bundle $V$ of type $p-q$ over $\cG$ admits a $\spin^c$-\emph{structure} if $w(V)=0$; in this case we say that $V$ is $\spin^c$. We also say that $V$ is $\KR$-oriented (following Hilsum-Skandalis terminology~\cite{Hilsum-Skandalis:Morphismes}, see also H. Schr\"oder's book~\cite{Schroder:KR-theory}).	\end{df}

Taking the involution of $\cG$ to be the trivial one, the next result is actually the groupoid equivariant analogue of Plymen's~\cite[Theorem 2.8]{Plymen:Strong_Morita}.

\begin{pro}~\label{pro:DD(Cl(V))-vs-DD(A_V)}
Let $V$ be a "real" Euclidean vector bundle of type $p-q$ over $\grpd$.
\begin{enumerate}
\item We have $DD(\cA_V)=w(V)$. Hence, $V$ is $\KR$-oriented if and only if $\cA_V$ is trivial.

\item If $p=q\in \NN^\ast$, then $DD(\Cl(V))=DD(\cA_V)=w(V)$. Therefore, if $p=q$, the following statements are equivalent:
\begin{itemize}
\item[(i)] $V$ is $KR$-oriented.
\item[(ii)] The Rg D-D bundle $\Cl(V)\To X$ is trivial.
\item[(iii)] The Rg D-D bundle $\cA_V\To X$ is trivial.
\end{itemize}
\end{enumerate}
\end{pro}

\begin{proof}
1. Since the map sending a "real" graded $\uc$-central extension onto its Dixmier-Douady class is a natural isomorphism (see~\cite{Moutuou:Brauer_Group}), we have a commutative diagram 
\begin{equation*}
	\xymatrix{\wRExt(O(p+q),\uc) \ar[rr]^{\FF(V)^\ast} \ar[d]_{dd}^{\cong} & & \wRExt(\cG,\uc) \ar[d]^{dd}_{\cong} \\ \check{H}R^1(O(p+q)_\bullet,\ZZ_2)\times \check{H}R^2(O(p+q)_\bullet,\uc) \ar[rr]^ -{\FF(V)^\ast \times \FF(F)^\ast} & & \check{H}R^1(\cG_\bullet,\ZZ_2)\times\check{H}R^2(\cG_\bullet,\uc)}	
	\end{equation*}
Hence, if $f:\cG[\cU]\To O(p+q)$ is a "real" homomorphism realising $\FF(V)$, then 
\[DD(\cA_V)=dd(f^\ast \cT_{p,q})=(f^\ast w_1(\theta^{p,q}),f^\ast w_2(\theta^{p,q}))=w(V),\]
where the first equality comes from the very definition of the "real" graded D-D bundle $\cA$, the second one follows from Example~\ref{ex:Stiefel-Whitney-theta_p-q}, and the last one is a simple interpretation on the construction of $w_1$ and $w_2$.	\\

2. If $p=q$, then $\Cl_{p,p}$ is the type $0$ "real" graded elementary $\cstar$-algebra $\cK(\hat{H})$, where $$\hat{H}=\CC^{2^{p-1}}\oplus \CC^{2^{p-1}}.$$ 
Identifying the "real" space $\RR^{p,p}$ with $\CC^p$ endowed with the coordinatewise complex conjugation, there is a degree conserving "real" representation
\[\lambda: \pin^c(p+p) \to \wU(\hat{H}),\]
induced by the "real" map $\CC^p\To \cL(\Lam^\ast \CC^p)\cong \cL(\hat{H})$ given by exterior multiplication (cf.~\cite{Kasparov:Operator_K}). This gives us a projective "real" representation 
\[
Ad_\lambda: O(p+p)\To \wPU(\hat{H})
\]
given by $Ad_\lambda (\g)(T):=\lambda(\tilde{\g})T\lambda(\tilde{\g})^{-1}$, for $T\in \cK(\hat{H})$, where $\tilde{\g}\in \pin^c(2p)$ is an arbitrary lift of $\g\in O(p+p)$, and where we have used the well-known identification $\wPU(\hat{H})\cong \Aut^{(0)}(\cK(\hat{H}))$. We thus have commutative diagrams

\begin{equation*}
\xymatrix{ \uc \ar[r] \ar@{=}[dd] & \pin^c(p+p) \ar[r] \ar[dd]^{\lambda} & O(p+p) \ar[dd]^{Ad_\lambda} \ar[rd] & \\
& & & \ZZ_2 \\
\uc \ar[r] & \wU(\hat{H}) \ar[r] & \wPU(\hat{H}) \ar[ur] }	
\end{equation*}
which yields to an equivalence of "real" graded $\uc$-central extensions 
\begin{equation}~\label{eq:Ad_lambda-E_K-vs-Tau_p-p}
	(Ad_\lambda)^\ast \EE_{\cK(\hat{H})} \sim \cT_{p,p} \in \wRExt(O(p+q),\uc)
\end{equation}
where $\EE_{\cK(\hat{H})}$ is the "real" graded $\uc$-central extension $(\wU(\hat{H}),deg)$ of the "real" groupoid $\xymatrix{\wPU(\hat{H})\dar[r]&\cdot}$.

Now let $(U_i,h_i)$ be a "real" trivialisation of $V$, with transition functions $\al_{ij}:U_{ij}\To O(p+p)$. Then $V$ is isomorphic to the Euclidean "real" vector bundle over $\cG$
\[
\coprod_{i}U_i\times \CC^p/_\sim \To X,
\]
where $(x,t)_i\sim (x,\al_{ij}(x)t)_j$, for $x\in U_{ij}$, endowed with the "real" $\cG$-action 
\[g\cdot [(s(g),t)]_{j_1}:= [(r(g),\frc_{(j_1,j_0)}(g)t)]_{j_0}, g\in U^1_{(j_0,j_1)}.\] 
Here $[(x,t)]_i$ denotes the class of $(x,t)_i\in U_i\times \CC^p$ in $\coprod_iU_i\times\CC^p/_\sim$.

Moreover, by universality of Clifford algebras (cf.~\cite[Chap.IV, \S4]{Karoubi:K-theory_Intro}), the "real" family of homeomorphisms $h_i:V_{|U_i}\To U_i\times \CC^p$ induces a "real" family of homeomorphisms $$\Cl(V)_{|U_i}\To U_i\times \Cl_{p,p}=U_i\times \cK(\hat{H})$$ with transition functions $Ad_\lambda\al_{ij}(\cdot):U_{ij}\To \wPU(\hat{H})$. Since the "real" action of $\cG$ on the Rg D-D bundle $\Cl(V)\To X$ is induced from the "real" action of $\cG$ on $V\To X$, it follows that $\Cl(V)$ is isomorphic to the "real" graded D-D bundle 
\[\coprod_iU_i\times \cK(\hat{H})/_\sim \]
where the equivalence relation is $(x,a)_i\sim (x,Ad_\lambda( \al_{ij}(x))a)_j$ for $x\in U_{ij}$, with the "real" $\cG$-action by graded automorphisms 
\[
g\cdot [(s(g),a)]_{j_1}:=[(r(g),Ad_\lambda (\frc_{(j_0,j_1)})a)]_{j_0}, g\in U^1_{(j_0,j_1)}.
\]
Therefore, if $P:\cG\To \wPU(\hat{H})$ is the generalised classifying morphism (~\cite[Section 66]{Moutuou:Brauer_Group}) for $\Cl(V)$, it corresponds to the class of $Ad_\lambda \frc$ in $\check{H}R^1(\cG_\bullet,\wPU(\hat{\cH}))$, thanks to~\cite[Proposition 7.4]{Moutuou:Brauer_Group}. Putting this in terms of generalised "real" homomorphisms, there is a commutative diagram in the category $\RG$ 
\[
\xymatrix{ & O(p+p) \ar[dd]^{Ad_\lambda} \\ \cG \ar[ur]^{\FF(V)} \ar[dr]_{P} & \\ & \wPU(\hat{H})}
\]
This combined with~\eqref{eq:Ad_lambda-E_K-vs-Tau_p-p} implies 
\[
P^\ast\EE_{\cK(\hat{H})}\sim \FF(V)^\ast \cT_{p,p}.
\]
Hence, $DD(\Cl(V))=dd([P^\ast \EE_{\cK(\hat{H})}])=dd([\FF(V)^\ast \cT_{p,p}]) =dd(\cA_V)=w(V)$, where the third and fourth equalities come from the first statement of the proposition.
\end{proof}

To see how things work in the general case, observe first that if $V_1, V_2$ are "real" Euclidean vector bundles of types $p_1-q_1$ and $p_2-q_2$, respectively, then $\Cl(V_1\oplus V_2)\To X$ is a "real" graded D-D bundle of type $(q_1+q_2)-(p_1+p_2)$ because $\Cl(V_1\oplus V_2)\cong \Cl(V_1)\hat{\otimes}_X\Cl(V_2)$ (cf.~\cite[\S2.15]{Kasparov:Operator_K} or~\cite{Atiyah-Bott-Schapiro:Clifford}).

Now let $V$ be a "real" Euclidean vector bundle of type $p-q$ over $\grpd$. Then we define \ $\tilde{V}:=V\oplus \id^{q,p}$, and $\widetilde{\Cl}(V) \in \wRBr_0(\cG)$ as the Rg D-D bundle of type $0$ defined by  $$\widetilde{\Cl}(V):=\Cl(\tilde{V})\To X,$$ with the obvious "real" $\cG$-action. Notice that this definition is a direct adaptation of~\cite[\S3.5]{Alekseev-Meinrenken:Dirac}. Moreover, we have the following

\begin{thm}~\label{thm:Calcul-de-Cl(V)}
Let $V$ be a "real" Euclidean vector bundle of type $p-q$ over $\grpd$. Then
\begin{eqnarray}
 DD(\Cl(V)) & = & (q-p, w_1(V),w_2(V)).
\end{eqnarray}	
\end{thm}

\begin{proof}
We have $DD(\widetilde{\Cl}(V))=DD(\cA_{\tilde{V}})$, thanks to Proposition~\ref{pro:DD(Cl(V))-vs-DD(A_V)} 2). Applying Lemma~\ref{lem:A_(sym)-vs-pro-A} and Lemma~\ref{lem:A_id=0}, we get $DD(\widetilde{\Cl}(V))=DD(\cA_V)$.
Furthermore, $\Cl(V)$ is clearly Morita equivalent to $\widetilde{\Cl}(V)\hat{\otimes}_X\Cl(\id^{p,q})$. Therefore, $$DD(\Cl(V))=DD(\widetilde{Cl}(V))+DD(\Cl(\id^{p,q}))=DD(\cA_V)+(q-p,0,0).$$ We conclude by applying Proposition~\ref{pro:DD(Cl(V))-vs-DD(A_V)} 1).	
\end{proof}

By using the fact $DD$ is a group homomorphism, we immediately deduce from the above theorem that

\begin{cor}
If $V$ and $V'$ are "real" Euclidean vector bundles over $\grpd$ then \\
$$w_1(V\oplus V')=w_1(V)+w_1(V'), \quad {\rm and\ } w_2(V\oplus V')=(-1)^{w_1(V)+w_1(V')}w_2(V)\cdot w_2(V').$$	
\end{cor}

\medskip 

\begin{proof}[Proof of Theorem~\ref{thm:main-thm-vector}]
As a consequence of Theorem~\ref{thm:Calcul-de-Cl(V)}, one has $$\KR^\ast_{\Cl(V)}(\cG^\bullet)\cong \KR^\ast_{\cA_V+\Cl_{p,q}}(\cG^\bullet),$$
so we conclude by using Kasparov product in $\KKR$-theory (recall that by definition we have $\KR^\ast_{\cA+\cB}(\cG^\bullet)=\KKR_{-\ast}(\CC,(\cA\hat{\otimes}_X\cB)\rtimes_r\cG)$).	
\end{proof}


\section{Thom isomorphism in twisted groupoid $K$-theory}

We start this section by some observations about $\spin^c$ "real" Euclidean vector bundles. Let
\[
\spin(p+q):=\pin(p+q)\cap \Cl_{p,q}^0
\]
 (cf.~\cite{Atiyah-Bott-Schapiro:Clifford, Karoubi:K-theory_Intro, Kasparov:Operator_K}). The restriction of the projection $\pin(p+q)\To O(p+q)$ induces a surjective "real" homomorphism $$\spin(p+q)\To SO(p+q)$$ with kernel $\ZZ_2$, where $SO(p+q)$ is equipped with the "real" structure induced from $O(p+q)$. Moreover, there is a "real" (trivially) graded $\uc$-central extension $\cT_{p,q}'$ 

\begin{equation*}
	\xymatrix{\uc \ar[r] & \spin^c(p+q) \ar[r] & SO(p+q) \ar[d] \\ & & \ZZ_2}
\end{equation*}
over the "real" groupoid $\xymatrix{SO(p+q)\dar[r]& \cdot}$, where $SO(p+q)\To \ZZ_2$ is the zero map, and where 
$$\spin^c(p+q):=\spin(p+q)\times_{\ZZ_2}\uc.$$

Now suppose $V$ is a "real" Euclidean vector bundle of type $p-q$ over $\grpd$. If $w_1(V)=0$, then $\FF(V)$ reduces to a generalised "real" homomorphism from $\cG$ to $\xymatrix{SO(p+q)\dar[r]& \cdot}$. So, $\cA_V$ comes from the "real" graded $\uc$-central extension $\FF(V)^\ast\cT_{p,q}'$. Moreover, $V$ being $\KR$-oriented means that $\FF(V)$ actually is a "real" $\spin^c$-principal bundle over $\grpd$, hence a generalised "real" homeomorphism from $\grpd$ to $\xymatrix{\spin^c(p+q)\dar[r]&\cdot}$.

The following result generalises the Thom isomorphism theorem in twisted orthogonal $\K$-theory known in the case of topological spaces (see Karoubi and Donovan~\cite{Donovan-Karoubi} and Karoubi~\cite{Karoubi:Twisted}) and in twisted $\K$-theory of bundle gerbes proved by A. Carey and B.-L. Wang in~\cite{Carey-Wang:Thom}.

\begin{thm}~\label{thm:Thom-isomorphism}
Let $\grpd$ be a locally compact Hausdorff second-countable "real" groupoid with "real" Haar system. Let $\pi:V\To X$ be a "real" Euclidean vector bundle of type $p-q$ over $\grpd$, and let $\cA\in \wRBr(\cG)$. Then there is a canonical group isomorphism

\begin{eqnarray}~\label{eq:Thom1}
\KR_{\pi^\ast\cA}^\ast((\pi^\ast\cG)^\bullet) &\cong &\KR_{\cA+\Cl(V)}(\cG^\bullet).
\end{eqnarray}
Furthermore, if \ $V$ is $\KR$-oriented, then there is a canonical isomorphism
\begin{eqnarray}~\label{eq:Thom2}
	\KR_{\pi^\ast\cA}^\ast((\pi^\ast\cG)^\bullet) &\cong & \KR^{\ast-p+q}_\cA(\cG^\bullet),
\end{eqnarray}
where as usual, the "real" groupoid $\xymatrix{\pi^\ast\cG\dar[r]& V}$ is the pullback of $\grpd$ via the projection $\pi$.
\end{thm}

\begin{proof}
From Theorem~\ref{thm:main-thm-vector} we have $\KR_{\cA+\Cl(V)}^\ast(\cG^\bullet)\cong \KR^{\ast-p+q}_{\cA+\cA_V}(\cG^\bullet)$; in particular if $V$ is $\spin^c$, $\cA_V=0$ and $\KR^\ast_{\cA+\Cl(V)}\cong \KR^{\ast-p+q}_\cA(\cG^\bullet)$, which implies that the isomorphism~\eqref{eq:Thom2} deduces from isomorphism~\eqref{eq:Thom1}. Let us show the latter. The $\KKR_{O(p+q)}$-equivalence $$\al_{p,q}\in \KKR(C_0(\RR^{p,q}),\Cl(\RR^{p,q}))$$ induces by functoriality in the category $\RG$ a $\KKR_\cG$-equivalence $$\FF(V)^\ast\al_{p,q}\in \KKR_{\cG,\ast}(C_0(V)\hat{\otimes}C_0(X),C_0(X;\Cl(V))).$$ Thus, by the identifications of $\cstar$-algebras with generalized "real" $\cG$-actions  $$C_0(X;\cA)\cong C_0(X)\hat{\otimes}_{C_0(X)}C_0(X;\cA), \quad {\rm and}$$ 
$$C_0(V;\pi^\ast\cA)\cong C_0(V)\hat{\otimes}_{C_0(X)}C_0(X;\cA),$$ we get a $\KKR_\cG$-equivalence 
\[\tilde{\al}_V\in \KKR_{\cG,\ast}(C_0(V;\pi^\ast \cA),C_0(X;\cA\hat{\otimes}_X\Cl(V))),\]
by taking $\al_V$ to be the Kasparov product of $\FF(V)^\ast\al_{p,q}$ with the canonical $\KKR_\cG$-equivalence $$\id_{C_0(X;\cA)}\in \KKR_\cG(C_0(X;\cA),C_0(X;\cA)).$$
Therefore, we obtain a $\KKR$-equivalence $$\al_V:=j_{\cG,red}(\tilde{\al}_V)\in \KKR_\ast(C_0(V;\pi^\ast \cA)\rtimes_r\cG,C_0(X;\cA\hat{\otimes}_X\Cl(V))\rtimes_r\cG),$$ where $j_{\cG,red}$ is the descent morphism; and we are done.
\end{proof}


\begin{bibdiv}\begin{biblist}
\bib{Alekseev-Meinrenken:Dirac}{article}{
  author={Alekseev, A.},
  author={Meinrenken, E.},
  title={Dirac structures and Dixmier-Douady bundles},
  journal={Int. Math. Res. Not. IMRN},
  date={2012},
  number={4},
  pages={904--956},
  issn={1073-7928},
  review={\MR {2889163}},
}

\bib{Atiyah-Bott-Schapiro:Clifford}{article}{
  author={Atiyah, M. F.},
  author={Bott, R.},
  author={Shapiro, A.},
  title={Clifford modules},
  journal={Topology},
  volume={3},
  date={1964},
  number={suppl. 1},
  pages={3--38},
  issn={0040-9383},
  review={\MR {0167985 (29 \#5250)}},
}

\bib{Blackadar:K-theory1}{book}{
  author={Blackadar, B.},
  title={\(K\)\nobreakdash -Theory for operator algebras},
  series={Mathematical Sciences Research Institute Publications},
  volume={5},
  publisher={Springer-Verlag},
  place={New York},
  date={1986},
  pages={viii+338},
  isbn={0-387-96391-X},
  review={\MRref {859867}{88g:46082}},
}

\bib{Carey-Wang:Thom}{article}{
  author={Carey, A. L.},
  author={Wang, B.-L.},
  title={Thom isomorphism and push-forward map in twisted $K$-theory},
  journal={J. K-Theory},
  volume={1},
  date={2008},
  number={2},
  pages={357--393},
  issn={1865-2433},
  review={\MR {2434190 (2009g:55005)}},
  doi={10.1017/is007011015jkt011},
}

\bib{Cuntz-Skandalis:Cones}{article}{
  author={Cuntz, J.},
  author={Skandalis, G.},
  title={Mapping cones and exact sequences in $KK$-theory},
  journal={J. Operator Theory},
  volume={15},
  date={1986},
  number={1},
  pages={163--180},
  issn={0379-4024},
  review={\MRref {816237}{88b:46099}},
}

\bib{Donovan-Karoubi}{article}{
  author={Donovan, P.},
  author={Karoubi, M.},
  title={Graded Brauer groups and $K$-theory with local coefficients},
  journal={Inst. Hautes \'Etudes Sci. Publ. Math.},
  number={38},
  date={1970},
  pages={5--25},
  issn={0073-8301},
  review={\MR {0282363 (43 \#8075)}},
}

\bib{Hilsum-Skandalis:Morphismes}{article}{
  author={Hilsum, M.},
  author={Skandalis, G.},
  title={Morphismes \(K\)\nobreakdash -orient\'es d'espaces de feuilles et fonctorialit\'e en th\'eorie de Kasparov \textup (d'apr\`es une conjecture d'A. Connes\textup )},
  language={French, with English summary},
  journal={Ann. Sci. \'Ecole Norm. Sup. (4)},
  volume={20},
  date={1987},
  number={3},
  pages={325--390},
  issn={0012-9593},
  review={\MRref {925720}{90a:58169}},
}

\bib{Husemoller:Fibre}{book}{
  author={Husem\"oller, D.},
  title={Fibre bundles},
  series={Graduate Texts in Mathematics},
  volume={20},
  edition={3},
  publisher={Springer-Verlag},
  place={New York},
  date={1994},
  pages={xx+353},
  isbn={0-387-94087-1},
  review={\MR {1249482 (94k:55001)}},
}

\bib{Jensen-Thomsen:Elements_KK}{book}{
  author={Jensen, K. K.},
  author={Thomsen, K.},
  title={Elements of $KK$-theory},
  series={Mathematics: Theory \& Applications},
  publisher={Birkh\"auser Boston Inc.},
  place={Boston, MA},
  date={1991},
  pages={viii+202},
  isbn={0-8176-3496-7},
  review={\MR {1124848 (94b:19008)}},
  doi={10.1007/978-1-4612-0449-7},
}

\bib{Karoubi:K-theory_Intro}{book}{
  author={Karoubi, M.},
  title={$K$\nobreakdash -theory},
  series={Classics in Mathematics},
  note={An introduction; Reprint of the 1978 edition; With a new postface by the author and a list of errata},
  publisher={Springer-Verlag},
  place={Berlin},
  date={2008},
  pages={xviii+308+e-7},
  isbn={978-3-540-79889-7},
  review={\MR {2458205 (2009i:19001)}},
  doi={10.1007/978-3-540-79890-3},
}

\bib{Karoubi:Twisted}{article}{
  author={Karoubi, M.},
  title={Twisted $K$-theory---old and new},
  conference={ title={$K$-theory and noncommutative geometry}, },
  book={ series={EMS Ser. Congr. Rep.}, publisher={Eur. Math. Soc., Z\"urich}, },
  date={2008},
  pages={117--149},
  review={\MR {2513335 (2010h:19010)}},
  doi={10.4171/060-1/5},
}

\bib{Kasparov:Operator_K}{article}{
  author={Kasparov, G. G.},
  title={The operator \(K\)\nobreakdash -functor and extensions of \(C^*\)\nobreakdash -algebras},
  language={Russian},
  journal={Izv. Akad. Nauk SSSR Ser. Mat.},
  volume={44},
  date={1980},
  number={3},
  pages={571--636, 719},
  issn={0373-2436},
  translation={ language={English}, journal={Math. USSR-Izv.}, volume={16}, date={1981}, number={3}, pages={513--572 (1981)}, },
  review={\MRref {582160}{81m:58075}},
}

\bib{Kasparov:Novikov}{article}{
  author={Kasparov, G. G.},
  title={Equivariant \(KK\)-theory and the Novikov conjecture},
  journal={Invent. Math.},
  volume={91},
  date={1988},
  number={1},
  pages={147--201},
  issn={0020-9910},
  review={\MRref {918241}{88j:58123}},
  doi={10.1007/BF01404917},
}

\bib{Kumjian:Fell}{article}{
  author={Kumjian, A.},
  title={Fell bundles over groupoids},
  journal={Proc. Amer. Math. Soc.},
  volume={126},
  date={1998},
  number={4},
  pages={1115--1125},
  issn={0002-9939},
  review={\MR {1443836 (98i:46055)}},
  doi={10.1090/S0002-9939-98-04240-3},
}

\bib{Kumjian-Muhly-Renault-Williams:Brauer}{article}{
  author={Kumjian, A.},
  author={Muhly, P. S.},
  author={Renault, J. N.},
  author={Williams, D. P.},
  title={The Brauer group of a locally compact groupoid},
  journal={Amer. J. Math.},
  volume={120},
  date={1998},
  number={5},
  pages={901--954},
  issn={0002-9327},
  review={\MR {1646047 (2000b:46122)}},
}

\bib{LeGall:KK_groupoid}{article}{
  author={Le Gall, P.-Y.},
  title={Th\'eorie de Kasparov \'equivariante et groupo\"\i des. I},
  language={French, with English and French summaries},
  journal={\(K\)\nobreakdash -Theory},
  volume={16},
  date={1999},
  number={4},
  pages={361--390},
  issn={0920-3036},
  review={\MRref {1686846}{2000f:19006}},
  doi={10.1023/A:1007707525423},
}

\bib{Moutuou:Brauer_Group}{article}{
  author={Moutuou, E. M.},
  title={The graded Brauer group of a groupoid with involution},
  status={eprint},
  note={\arxiv {1202.2057}},
  date={2012},
}

\bib{Moutuou:Real.Cohomology}{article}{
  author={Moutuou, E. M.},
  title={On groupoids with involutions and their cohomology},
  status={eprint},
  note={\arxiv {1202.0155}},
  date={2012},
}

\bib{Moutuou:Thesis}{thesis}{
  author={Moutuou, E. M.},
  title={Twisted groupoid $KR$--Theory},
  type={Ph.D. thesis},
  institution={Universit\'e de Lorraine - Metz, and Universit\"at Paderborn},
  eprint={http://www.theses.fr/2012LORR0042},
  date={2012},
}

\bib{Plymen:Strong_Morita}{article}{
  author={Plymen, R. J.},
  title={Strong Morita equivalence, spinors and symplectic spinors},
  journal={J. Operator Theory},
  volume={16},
  date={1986},
  number={2},
  pages={305--324},
  issn={0379-4024},
  review={\MR {860349 (88d:58112)}},
}

\bib{Schroder:KR-theory}{book}{
  author={Schr{\"o}der, H.},
  title={$K$-theory for real $C^*$-algebras and applications},
  series={Pitman Research Notes in Mathematics Series},
  volume={290},
  publisher={Longman Scientific \& Technical},
  place={Harlow},
  date={1993},
  pages={xiv+162},
  isbn={0-582-21929-9},
  review={\MR {1267059 (95f:19006)}},
}

\bib{Skandalis:Remarks_KK}{article}{
  author={Skandalis, G.},
  title={Some remarks on Kasparov theory},
  journal={J. Funct. Anal.},
  volume={56},
  date={1984},
  number={3},
  pages={337--347},
  issn={0022-1236},
  review={\MRref {743845}{86c:46085}},
}

\bib{Tu:BC_groupoids}{article}{
  author={Tu, J.-L.},
  title={The Baum--Connes conjecture for groupoids},
  conference={ title={$C^*$\nobreakdash -algebras}, address={M\"unster}, date={1999}, },
  book={ publisher={Springer}, place={Berlin}, },
  date={2000},
  pages={227--242},
  review={\MRref {1798599}{2001j:46109}},
}

\bib{Tu:Twisted_Poincare}{article}{
  author={Tu, J.-L.},
  title={Twisted $K$\nobreakdash -theory and Poincar\'e duality},
  journal={Trans. Amer. Math. Soc.},
  volume={361},
  date={2009},
  number={3},
  pages={1269--1278},
  issn={0002-9947},
  review={\MRref {2457398}{}},
}

\bib{Tu-Xu-Laurent-Gengoux:Twisted_K}{article}{
  author={Tu, J.-L.},
  author={Xu, Ping},
  author={Laurent-Gengoux, C.},
  title={Twisted $K$\nobreakdash -theory of differentiable stacks},
  language={English, with English and French summaries},
  journal={Ann. Sci. \'Ecole Norm. Sup. (4)},
  volume={37},
  date={2004},
  number={6},
  pages={841--910},
  issn={0012-9593},
  review={\MRref {2119241}{2005k:58037}},
}


 \end{biblist}\end{bibdiv}
\end{document}